\documentclass[a4paper,twoside,11pt]{amsart}

\usepackage[margin=3cm]{geometry}
\usepackage{graphicx}
\usepackage{amssymb}
\usepackage{mathrsfs}
\usepackage{enumitem}
\usepackage{array,multirow}
\usepackage{mathtools}
\usepackage[usenames,table,dvipsnames]{xcolor}
\usepackage{dsfont}
\usepackage{calligra}
\DeclareMathAlphabet{\mathcalligra}{T1}{calligra}{m}{n}
\usepackage{tikz}
\usepackage{tikz-cd}
\usepackage{etoolbox}
\usepackage{float}
\usepackage{multicol}
\usepackage{rotating}
\usepackage{arydshln}
\usepackage{caption}
\usepackage{subcaption}
\usepackage{mathdots}
\usepackage{threeparttable}
\usepackage{dashbox}

\usetikzlibrary{
  shapes.misc,
  decorations.markings,
  calc,
  automata,
  patterns,
  decorations,
  decorations.pathmorphing,
  fadings,
  positioning
}

\usepackage{hyperref}
\hypersetup{
  breaklinks    = true,
  colorlinks    = true,
  citecolor     = blue,
  linkcolor     = blue,
  filecolor     = magenta,
  urlcolor      = black
}
\expandafter\def\csname ver@etex.sty\endcsname{3000/12/31}  
\usepackage{autonum}        

\usepackage{soul}

\makeatletter

\newcommand{\Rmnum}[1]{\expandafter\@slowromancap\romannumeral #1@}
\makeatother

\newcommand{\To}[1]{\rule{0pt}{#1ex}}           
\newcommand{\Bo}[1]{\rule[-#1ex]{0pt}{0pt}}     

\newtheorem{theorem}{Theorem}[section]

\newtheorem{proposition}[theorem]{Proposition}
\newtheorem{corollary}[theorem]{Corollary}

\newtheorem{remark}[theorem]{Remark}
\newtheorem{definition}[theorem]{Definition}

\makeatletter
\def\@tocline#1#2#3#4#5#6#7{%
  \ifnum #1>\c@tocdepth
  \else
    \par\addpenalty\@secpenalty\addvspace{#2}%
    \begingroup
      \hyphenpenalty\@M
      \@ifempty{#4}{%
        \@tempdima\csname r@tocindent\number#1\endcsname\relax
      }{%
        \@tempdima#4\relax
      }%
      \parindent\z@%
      \leftskip#3\relax
      \advance\leftskip\@tempdima\relax
      \rightskip\@pnumwidth plus4em
      \parfillskip -\@pnumwidth
      #5\leavevmode\hskip -\@tempdima
      \ifcase #1\or\or \hskip 1em \or \hskip 2em \else \hskip 3em \fi%
      #6\nobreak\relax
      \dotfill\hbox to\@pnumwidth{\@tocpagenum{#7}}\par
      \nobreak
    \endgroup
  \fi
}
\makeatother

\title{Defining equations of $7$-dimensional model CR hypersurfaces}

\author{Jan Gregorovi\v{c}}
\address{Jan Gregorovi\v{c},
	Department of Mathematics and Statistics 
	Masaryk University,
	Kotl\'{a}\v{r}-sk\'{a} 2,
	611 37 Brno,
	Czech Republic, Department of Mathematics, Faculty of Science, University of Ostrava, 701 03 Ostrava, Czech Republic, and Institute of Discrete Mathematics and Geometry, TU Vienna, Wiedner Hauptstrasse 8-10/104, 1040 Vienna, Austria. ORCID ID: 0000-0002-0715-7911}\email{ jan.gregorovic@seznam.cz}

 \author{David Sykes}
\address{David Sykes,
	Department of Mathematics and Statistics, 
	Masaryk University,
	Kotl\'{a}\v{r}sk\'{a} 2,
	611 37 Brno,
	Czech Republic. ORCID ID: 0000-0002-6928-3753}\email{ sykes@math.muni.cz}

\makeatletter
\@namedef{subjclassname@2020}{\textup{2022} Mathematics Subject Classification}
\makeatother
\thanks{The authors were supported by the GACR grant GA21-09220S. J.G. was also supported by Austrian Science Fund (FWF): P34369}

\begin{document}
\begin{abstract}
We study CR hypersurfaces in $\mathbb{C}^4$ that are Levi degenerate with constant rank Levi form, and moreover finitely nondegenerate. Each of these can be described as a deformation of a model CR hypersurface by adding terms of higher natural weighted order to the model's defining equation. We obtain a complete normal form for models of real analytic uniformly $2$-nondegenerate CR hypersurfaces in $\mathbb{C}^4$, and present a detailed study of their local invariants. The normal form illustrates that $2$-nondegenerate models in $\mathbb{C}^4$ comprise a moduli space parameterized by two univariate holomorphic functions, which is in sharp contrast to the well known Levi-nondegenerate setting and the more recently discovered behavior of $2$-nondegenerate structures in $\mathbb{C}^3$. In further contrast to these previously studied settings, we demonstrate that not all $2$-nondegenerate structures in $\mathbb{C}^4$ arise as perturbations of homogeneous models. We derive defining equations for the homogeneous $2$-nondegenerate models, a set of $9$ structures, and find explicit formulas for their infinitesimal symmetries.
\end{abstract}
\maketitle
\tableofcontents
\section{Introduction}

At each point of a real analytic hypersurface-type CR manifold $M$ with Levi form of constant rank, there is a canonically associated CR hypersurface germ distinguished by admitting a natural weighted homogeneous defining equation described in \cite[Theorem 1.1]{GKS2024}. Our aim is to find a complete normal form for these associated CR hypersurfaces in $\mathbb{C}^4$, solving their local equivalence problem. We achieve this classification for the structures associated with CR hypersurfaces on which all integer-valued invariants of \cite[Section 11.1]{baouendi1999real} are constant. This a generic property that is equivalent to \emph{uniform $k$-nondegeneracy} of \cite[Definition 2.1]{GKS2024} for finitely nondegenerate structures, while refining the property of being \emph{everywhere holomorphically degenerate} for holomorphically degenerate structures.

In the Levi nondegenerate case, the associated CR hypersurface germs are locally equivalent to the usual quadric models 
$$
\Re(w)=\vert z_1\vert^2+\vert z_2\vert^2\pm\vert z_3\vert^2
$$
at $0$ in coordinates $(w,z_1,z_2,z_3)\in\mathbb{C}^4$.

The uniform $k$-nondegenerate hypersurfaces in $\mathbb{C}^4$ with Levi form of constant rank $1$ are necessarily $3$-nondegenerate and the associated CR hypersurface germs are locally equivalent to
\begin{equation}\label{eqn: lightconemodel}
   \Re(w)=\frac{|z_1|^2+\Re(\zeta\overline{z_1}^2)}{1-|\zeta|^2} 
\end{equation}
at $0$ in coordinates $(w,z_1,z_2,\zeta)\in\mathbb{C}^4$, which coincides with the (locally) unique $2$-nondegenerate model in $\{z_2=0\}\cong\mathbb{C}^3.$

For the uniformly holomorphically degenerate hypersurfaces in $\mathbb{C}^4$, the associated CR hypersurface germs are locally equivalent to one of
$$
\Re(w)=|z_1|^2\pm|z_2|^2,\quad \Re(w)=\frac{|z_1|^2+\Re(\zeta\overline{z_1}^2)}{1-|\zeta|^2},\quad 
\Re(w)=|z_1|^2,
\quad\mbox{ and }\quad
\Re(w) = 0
$$
at $0$ in coordinates $(w,z_1,z_2,\zeta)\in\mathbb{C}^4$.

 While the above observations, covering all cases with Levi form rank different from $2$, are straightforward consequences of  \cite[Theorem 1.1]{GKS2024}, the remaining class of rank $2$ Levi form structures is much more intricate. Uniform $k$-nondegenerate hypersurfaces in $\mathbb{C}^4$ with Levi form of constant rank $2$ are necessarily $2$-nondegenerate and we use the terminology \emph{$2$-nondegenerate models} for their associated CR hypersurface germs. The classification of $2$-nondegenerate models is a nontrivial task, because we know from \cite[Theorem 1.2]{GKS2024} that the moduli space of the $2$-nondegenerate models has functional freedom. To achieve the complete normal form of $2$-nondegenerate models in $\mathbb{C}^4$, one needs to study their local invariants in considerable detail. In Proposition \ref{degree 0 symmetry bound}, we classify the values of their \emph{modified symbol} invariants, which are important invariants as they completely determine the $2$-nondegenerate model in the cases where the model is homogeneous \cite[Lemma 3.3]{GKS2024}. This also allows us to produce the defining equations for all homogeneous $2$-nondegenerate models in $\mathbb{C}^4$ (Table \ref{table of defining equations}).

Starting with a real analytic hypersurface-type CR manifold $M$ in $\mathbb{C}^4$ with Levi form of constant rank $2$  that is $2$-nondegenerate at $x\in M$, we show in \cite[Section 6]{GKS2024} that there are holomorphic coordinates $(w,z,\zeta)$ centered at $x$ such that $M$ is given by a defining equation of the form
\begin{align}\label{general expansion}
    \Re(w)&=z^T\mathbf{H}\overline{z}+\Re\left(\overline{z}^T\mathbf{H}^T\mathbf{S}(\zeta)\mathbf{H}\overline{z}\right)+O(|\zeta|^2)+O(3),
\end{align}
where $z=(z_1,z_2)^T\in \mathbb{C}^2$, $\zeta\in\mathbb{C}$, $\mathbf{H}$ is a $2\times 2$ Hermitian matrix, $\mathbf{S}(\zeta)$ is a symmetric matrix of holomorphic functions in $\zeta$ (defined in a neighborhood of $0$), $O(|\zeta|^2)$ are terms divisible by $|\zeta|^2$, and $O(3)$ are terms of weighted degree greater that $2$ for the weights 
\begin{align}\label{grading convention}
  \mathrm{wt}(w)=2, \quad \mathrm{wt}(z)=1,\quad \mathrm{wt}(\zeta)=0 .
\end{align}
 The matrix $\mathbf{H}$ represents the nondegenerate part of the Levi form at the origin, $\mathbf{S}_\zeta(0):=\frac{\partial}{\partial \zeta} \mathbf{S}(0)\neq 0$ is the $2$-nondegeneracy condition, and the $2$-nondegenerate model is completely determined by the pair ($\mathbf{H}$,$\mathbf{S}(\zeta)$). Moreover, it is shown in \cite[Proposition 6.2]{GKS2024} that the $2$-nondegenerate model is the weighted homogeneous part of \eqref{general expansion} and takes the form
\begin{equation}\label{gen def fun}
\Re(w)=(z_1,z_2)H(\zeta,\overline{\zeta})(\overline{z_1},\overline{z_2})^T+\Re\left((\overline{z_1},\overline{z_2})S(\zeta,\overline{\zeta})(\overline{z_1},\overline{z_2})^T\right),
\end{equation}
where
\begin{align}\label{H from S2 formula}
H(\zeta,\overline{\zeta})=\frac12(\mathbf{H}(\mathrm{Id}-\overline{\mathbf{S}(\zeta)}\mathbf{H}^T\mathbf{S}(\zeta)\mathbf{H})^{-1}+(\mathrm{Id}-\mathbf{H}\overline{\mathbf{S}(\zeta)}\mathbf{H}^T\mathbf{S}(\zeta))^{-1}\mathbf{H})
\end{align}
and 
\begin{align}\label{S from S2 formula}
S(\zeta,\overline{\zeta})=\mathbf{H}^T(\mathrm{Id}-\mathbf{S}(\zeta)\mathbf{H}\overline{\mathbf{S}(\zeta)}\mathbf{H}^T)^{-1}\mathbf{S}(\zeta)\mathbf{H}.
\end{align}

Since each point $x$ of the CR hypersurface $M$ is associated with the germ of a $2$-nondegenerate model of the general form \eqref{gen def fun}, we want to classify the $2$-nondegenerate models in such coordinates up to local equivalence at the origin. Achieving this classification, we find a unique pair ($\mathbf{H}$,$\mathbf{S}(\zeta)$) for each local equivalence class, described across several propositions, which we summarize in the following theorem.

\begin{theorem}\label{main theorem}
    For every germ of a $2$-nondegenerate model hypersurface in $\mathbb{C}^4$, there are local holomorphic coordinates centered at $0$ such that it is defined by \eqref{gen def fun}, \eqref{H from S2 formula}, and \eqref{S from S2 formula} with parameters $\mathbf{H}$ and $\mathbf{S}(\zeta)$ given by one and only one of the formulas presented in Propositions \ref{table4prop3}, \ref{table4prop2}, \ref{table4prop1}, and \ref{table4prop}.
\end{theorem}

We organize the proof into several cases presented in the subsections of Section \ref{A complete normal form}. The proof starts with the partial normal form for $2$-nondegenerate models of \cite[Theorem 6.5]{GKS2024}. We sort the partial normal forms into cases based on a classification of certain geometric CR invariants associated with the $2$-nondegenerate CR hypersurfaces, and then the proof branches according to the values of these invariants. The CR invariants in question are the \emph{bigraded CR symbol} introduced in \cite{porter2021absolute}, the \emph{modified CR symbol} introduced in \cite{sykes2023geometry} and generalized in \cite{GKS2024} and the \emph{obstructions to the first order constancy} (of the bigraded symbol) introduced in \cite{GKS2024}.

In Section \ref{sec bigraded symbol}, we review the definitions and how to compute these CR invariants for the $2$-nondegenerate models in $\mathbb{C}^4$ following our constructions in \cite[Section 6]{GKS2024}, in detail. This in particular, allows us to completely classify the normal forms of these CR invariants that are realized by $2$-nondegenerate CR hypersurfaces in $\mathbb{C}^4$, Proposition \ref{degree 0 symmetry bound}. This classification fixes $\mathbf{H}$, $\mathbf{S}_\zeta(0)$ and $\mathbf{S}_{\zeta\zeta}(0)$ of the normal form and leaves (according to \cite[Theorem 6.5]{GKS2024}) only freedom in a finite dimensional Lie subgroup of $\mathrm{CU}(\mathbf{H})\cap G_{0,0}$ (depending on $\mathbf{H}$, $\mathbf{S}_\zeta(0)$ and $\mathbf{S}_{\zeta\zeta}(0)$, see \eqref{CUdef} and \eqref{G00def}).

Regarding Section \ref{A complete normal form} in more detail, we resolve, through case by case analysis, the action of $\mathrm{CU}(\mathbf{H})\cap G_{0,0}$ on the remainder of the expansion of  $\mathbf{S}(\zeta)$ to obtain the normal forms. This allows us to classify the subgroup of CR symmetries of $2$-nondegenerate models belonging to $\mathrm{CU}(\mathbf{H})\cap G_{0,0}$. In particular, there are distinguished models realizing the modified CR symbols for which this subgroup is maximal  among all other $2$-nondegenerate CR hypersurfaces with this modified symbol (at a point), Corollary \ref{C4 degree zero bound}. We investigate these distinguished realizations of modified symbols in Section \ref{7-dimensional realizations} in detail, from which we obtain the defining equations for all homogeneous $2$-nondegenerate models, Proposition \ref{all hom models prop}. 

In table \ref{table of defining equations}, we summarize the new defining equations we obtain along with the previously know defining equations for the homogeneous $2$-nondegenerate CR hypersurfaces in $\mathbb{C}^4$.  For $7$ of them, equivalent defining equations to those provided in this article have been found previously, while $2$ are completely new.

\begin{table}[tbp]
\scalebox{.9}{
\begin{threeparttable}
\begin{tabular}{|l|c|c|c|}
\hline
\multirow{2}*{Label}& Known Defining Equation & \multirow{2}*{ISAD} & \multirow{2}*{References/Comments}\\
\cdashline{2-2}
& Defining Equation's Derived in \S\ref{Approaches to Hypersurface Realization}& & \\
\hline
\hline
\multirow{2}*{\To{6}($2$-nond.) Type \Rmnum{7}} & \Bo{2}$x_0=x_1x_2+x_1^2x_3$ 
&
\multirow{2}*{\To{5}16}
&
\To{3.4}\Bo{3.4}\multirow{2}*{ \parbox{4.2cm}{\flushleft The first equation appears in \cite[Theorem 2]{mozey2000} and the second is in \cite[Section 5, Example 1 with $n=3$]{labovskii1997dimensions} and \cite{beloshapka2022modification,porter2021absolute}.}}\\\cdashline{2-2}
\To{4.5}\Bo{4.5}& $\Re(w)=z_1\overline{z_2}+\overline{z_1}z_2+\Re(\zeta\overline{z_2}^2)$&&
\\\hline
\multirow{2}*{\To{5}($2$-nond.) Type \Rmnum{6}} & \To{4.5}\Bo{3}$\displaystyle \Re(w)=\frac{\Im\left(4z_1\overline{z_2}+\overline{z_3}(z_1^2+z_2^2)\right)}{1+\vert z_3\vert^2}$ 
&
\multirow{2}*{\To{5}15}
&
\multirow{2}*{\parbox{4.2cm}{\flushleft Studied in \cite{porter2021absolute,santi2020}, and this formula is derived in \cite[Section 5.4]{gregorovic2021equivalence}.}}\\\cdashline{2-2}
&$\displaystyle\Re(w)=\frac{|z_1|^2-|z_2|^2-2\Re(\zeta\overline{z_1}\overline{z_2})}{1+|\zeta|^2}$\To{4.6}\Bo{3.6}&&
\\\hline
\To{3}\Bo{3}\multirow{2}*{\To{6}($2$-nond.) Type \Rmnum{5}.A} & \To{4}\Bo{2.6}$\displaystyle x_0=\frac{x_1^2+x_2^2}{1+x_3},\,\, x_3>-1$ 
&
\multirow{2}*{\To{6}15}
&
\multirow{4}*{\To{11}\parbox{4.2cm}{\flushleft These are $7$-dimensional tubes over null cones of symmetric forms in $\mathbb{R}^4$, studied in \cite{porter2021absolute,santi2020}, and these formulas are derived in \cite[Section 5.3]{gregorovic2021equivalence}.}
}\\\cdashline{2-2}
&\To{4.2}\Bo{3}$\displaystyle\Re(w)=\frac{|z_1|^2+|z_2|^2+2\Re(\zeta\overline{z_1}\overline{z_2})}{1-|\zeta|^2}$&&
\\\cline{1-3}
\multirow{2}*{\To{5}($2$-nond.) Type \Rmnum{5}.B} & \To{4}\Bo{2.6}$\displaystyle x_0=\frac{x_1^2-x_2^2}{1+x_3},\,\, x_3>-1$ 

&
\multirow{2}*{\To{5}15}
&\\\cdashline{2-2}
&\To{4.2}\Bo{3}$\displaystyle\Re(w)=\frac{2\Re(z_1\overline{z_2})+2\Re(\zeta\overline{z_1}\overline{z_2})}{1-|\zeta|^2}$&&
\\\hline
\multirow{2}*{\To{6.5}($2$-nond.) Type \Rmnum{4}.A} &\To{4}\Bo{3.6}$\displaystyle x_0=\frac{x_1^2}{1+x_3}+x_2^2,\,\, x_3>-1$ 
&
\multirow{2}*{\To{6.5}10}
&
\multirow{4}*{\To{14.4} \parbox{4.2cm}{\flushleft These are studied in \cite{porter2021absolute,santi2020}. In \cite{porter2021absolute} they appear among the structures labeled as \emph{weakly non-nilpotent flat}.
}
}\\\cdashline{2-2}
&\To{4.2}\Bo{3}$\displaystyle\Re(w)=\frac{|z_1|^2+\Re(\zeta\overline{z_1}^2)}{1-|\zeta|^2}+|z_2|^2$&&
\\\cline{1-3}
\multirow{2}*{\To{6.5}($2$-nond.) Type \Rmnum{4}.B} & \To{4}\Bo{3.6}$\displaystyle x_0=\frac{x_1^2}{1+x_3}-x_2^2,\,\, x_3>-1$ 
&
\multirow{2}*{\To{6.5}10}
&\\\cdashline{2-2}
&\To{4.2}\Bo{3}$\displaystyle\Re(w)=\frac{|z_1|^2+\Re(\zeta\overline{z_1}^2)}{1-|\zeta|^2}-|z_2|^2$&&
\\\hline
($2$-nond.) Type \Rmnum{3} &
given by \eqref{def eqn type 3}
&
9
&
\\\hline
\To{4}\Bo{4.4}\multirow{2}*{\To{7}($2$-nond.) Type \Rmnum{2}} &$\displaystyle x_0=\frac{x_1x_2}{1+x_3}+\frac{x_2^2x_3}{(1+x_3)^2},\,\, x_3>-1$ 
&
\multirow{2}*{\To{7}8}
&
\multirow{2}*{\parbox{4.2cm}{This is equivalent to a tube over the hypersurface in $\mathbb{R}^4$ given by formula (1) in \cite[Theorem 2]{mozey2000}.}}\\\cdashline{2-2}
&\To{3.8}\Bo{3.2} given by \eqref{def eqn type 2 b}&&
\\\hline
($2$-nond.) Type \Rmnum{1} &
given by  \eqref{def eqn type 1 b}
&
8
&
\\\hline
\end{tabular}
\caption{Defining equations for homogeneous $2$-nondegenerate models in $\mathbb{C}^4$, given with $x_j$ denoting the real part of a complex variable. The labels \emph{Type \Rmnum{1}-\Rmnum{7}} refer to the labeling of the $2$-nondegenerate structures assigned in the classification of \cite{SykesHomogeneous} (compare with \cite[Table 1]{SykesHomogeneous}). The column labeled \emph{ISAD} displays the structures' infinitesimal symmetry algebra dimensions.} \label{table of defining equations}
\end{threeparttable}
}
\end{table}

 We record the defining equations for the two new models here. The first of the new models is given by 
\begin{align}\label{def eqn type 3}
\Re(w)&=2\frac{\Re\left(2|z_1|^2+\zeta\overline{z_1}^2\right)}{1-|\zeta|^2}\\
&\quad + \frac{\Re\left(-9|z_1|^2-3|z_2|^2-6\overline{z_1}^2\zeta+4\sqrt{3}z_1\overline{z_2}\zeta+2\sqrt{3}\overline{z_1z_2}\zeta^2\right)}{\left( 1-|\zeta|^2\right)^2}+\\
&\quad + 2\frac{\Re\left(3|z_1|^2+|z_2|^2+3\overline{z_1}^{2}\zeta-2\sqrt{3}z_1\overline{z_2}\zeta+\overline{z_2}^2\zeta^3-2\sqrt{3}\overline{z_1z_2}\zeta^2\right)}{\left(1-|\zeta|^2\right)^3},
\end{align}
defining a CR hypersurface with the intriguing property that the symmetry algebra of the dynamical Legendrian contact structure naturally induced on its (local) Levi leaf space in \cite{SykesHomogeneous} is the split real form of the complex Lie algebra $\mathrm{Lie}(G_2)$.

The other new model is given by
\begin{align}
   \Re(w)=&z_1\overline{z_2}+\overline{z_1}z_2+\frac{\Re\left(-(1+i)(\overline{\zeta}+\overline{\zeta}^2)z_1^2 -2\sqrt{2}iz_1z_2\overline{\zeta}^2)+(1-i)(\overline{\zeta}+\overline{\zeta}^2)z_2^2\right)}{(2\zeta+1)(2\overline{\zeta}+1)}\\
   &+\frac{2|\zeta|^2}{(2\zeta+1)(2\overline{\zeta}+1)}\bigg(-2\Re\left(z_1\overline{z_2}((1+i)|\zeta|^2+2i\Re(\zeta)+i) \right)\label{def eqn type 1 b}\\
&\hspace{3.7cm}+\sqrt{2}|z_1|^2(|\zeta|^2+\Re((1-i)\zeta))+\sqrt{2}|z_2|^2(|\zeta|^2+\Re((1+i)\zeta))\\
&\hspace{3.7cm}+\Re\left(z_1^2\left((i-1)|\zeta|^2-2\overline{\zeta}-1-i\right) \right)-2\sqrt{2}\Re\left(iz_1z_2(|\zeta|^2+\overline{\zeta})\right)\\
&\hspace{3.7cm}+\Re\left(z_2^2((1+i)|\zeta|^2+2\overline{\zeta}+1-i)\right)\bigg).
\end{align} 
We also record here the new defining equation given by 
\begin{align}\label{def eqn type 2 b}
   \Re(w)&=\frac{z_1\overline{z_2}+\overline{z_1}z_2+2\Re\left(z_1z_2\overline{\zeta}\right)+\Re\left(z_2^2(\overline{\zeta}^2+\overline{\zeta})\right)}{1-|\zeta|^2}+\frac{|\zeta|^2}{(1-|\zeta|^2)^2}(2|z_2|^2\Re\left(1+\overline{\zeta}\right)\\
    &\quad\quad+\Re\left(z_2^2(|\zeta|^2+\overline{\zeta}^2+2\overline{\zeta})\right)).
\end{align}
for the type \Rmnum{2} structure in Table \ref{table of defining equations} because it does not fit well within table.

It is notable that for $3$-nondegenerate structures in $\mathbb{C}^4$ there is a unique bigraded symbol, and moreover Kruglikov--Santi proved in \cite{kruglikov2023} that there is a locally unique $3$-nondegenerate homogeneous hypersurface in $\mathbb{C}^4$. Expressed in the partial normal form of \cite[Theorem 1.1]{GKS2024}, its defining equation has a naturally distinguished weighted homogeneous part given by \eqref{eqn: lightconemodel}, and its $3$-nondegenerate character can only be detected from the higher order terms.

Lastly, Section \ref{Infinitesimal symmetry algebras of homogeneous models} presents the full descriptions of the entire symmetry algebras for the type \Rmnum{1}, \Rmnum{2}, and \Rmnum{3} structures of Table \ref{table of defining equations}, focusing on just these three because the symmetry algebra descriptions for the other models in Table \ref{table of defining equations} can be found from the references cited within the table.

\section{Invariants of CR hypersurfaces: bigraded symbols and modified symbols}\label{Symbols of CR hypersurfaces}\label{sec bigraded symbol}

Let us now consider a $2$-nondegenerate model $M_0$ given by \eqref{gen def fun}, \eqref{H from S2 formula}, and \eqref{S from S2 formula}. We will determine the basic invariants of $M_0$. We choose the following frame $\phi=(g,f_1,f_2,e,\overline{f_1},\overline{f_2},\overline{e})$ of $\mathbb{C}TM_0:$
\begin{align}
g:=&\frac{\partial}{\partial \Im(w)},\\
f_a:=&\frac{\partial}{\partial z_a}-i(H_a(\zeta,\overline{\zeta})\overline{z}+S_a(\zeta,\overline{\zeta})z)\frac{\partial}{\partial \Im(w)},\\
e:=&\frac{\partial}{\partial \zeta}-i\left(z^TH_{\zeta}(\zeta,\overline{\zeta})\overline{z}+\Re(z^TS_{\zeta}(\zeta,\overline{\zeta})z)\right)\frac{\partial}{\partial \Im(w)}\\
&-\sum_{\overline{b},c}\left(z^TH_{\zeta,\overline{b}}(\zeta,\overline{\zeta})+\overline{z}^T\overline{S_{\overline{\zeta},\overline{b}}(\zeta,\overline{\zeta})}\right)(H(\zeta,\overline{\zeta})^{-1})_{c,\overline{b}}f_c,
\end{align}
where the notation $H_a(\zeta,\overline{\zeta})$ and $S_a(\zeta,\overline{\zeta})$ denotes the index $a$ row of $H(\zeta,\overline{\zeta})$ and $S(\zeta,\overline{\zeta})$, respectively, $(H(\zeta,\overline{\zeta})^{-1})_{c,\overline{b}}$ denotes the $(c,b)$ entry of $H(\zeta,\overline{\zeta})^{-1}$, and $H_{\zeta,\overline{b}}(\zeta,\overline{\zeta})$ (resp. $S_{\zeta,\overline{b}}(\zeta,\overline{\zeta})$) denotes the index $b$ column of $H_{\zeta}(\zeta,\overline{\zeta})$ (resp. $S_{\zeta}(\zeta,\overline{\zeta})$).

The vector fields $f_1$, $f_2$, and $e$ span a CR structure  $\mathcal{H}\subset \mathbb{C}TM_0$ on $M_0$, that is, $\mathcal{H}$ is a $3$-dimensional complex, integrable (i.e., $[\mathcal{H},\mathcal{H}]\subset \mathcal{H}$) subbundle of $\mathbb{C}TM_0$ with   $\mathcal{H}\cap \overline{\mathcal{H}}=0$. 

Levi form of $M_0$ is defined by
\[
\mathcal{L}^{1}\left(X_p,Y_p\right):=\frac{1}{2i}\left[X,\overline{Y}\right]_p\pmod{\mathcal{H}\oplus\overline{\mathcal{H}}}
\,\,\forall\,p\in M_0,
\]
where $X,Y$ are sections of $\mathcal{H}$ extending $X_p,Y_p.$
The Levi form $\mathcal{L}^1$ can be represented with respect to the frame $\phi$ by the matrix
\begin{align}\label{key example Lform}
\left(
\begin{array}{cc}
H(\zeta,\overline{\zeta}) & 0 \\
0  & 0  \\
\end{array}
\right),
\end{align}
where the off diagonal zeros follow from the definition of the frame, while the diagonal zero is consequence of \eqref{H from S2 formula} and \eqref{S from S2 formula}. Further consequences of \eqref{H from S2 formula} and \eqref{S from S2 formula} that we will use are
\begin{align}
H(0,0)&=\mathbf{H},\
S(0,0)=0,\
H_{\zeta}(0,0)=0,\\
S_{\overline{\zeta}}(0,0)=0,\
S_{\zeta\overline{\zeta}}(0,0)&=0,\
S_{\zeta}(0,0)=\mathbf{H}^T\mathbf{S}_\zeta(0)\mathbf{H},\
S_{\zeta\zeta}(0,0)=\mathbf{H}^T\mathbf{S}_{\zeta\zeta}(0)\mathbf{H}.
\end{align}

In particular, the vector field $e$ spans the Levi kernel $\mathcal{K}$, i.e., the left-kernel of $\mathcal{L}^1$ in $\mathcal{H}$. As we know from \cite[Theorem 1.2]{GKS2024}, there are infinitesimal symmetries transversal to the leaves of the Levi kernel. Therefore, we will focus the computations of invariant at the points $p\in M_0$ with $\Im(w)=0,z_1=0,z_2=0$, because this is the leaf of span of $e,\overline{e}$ through $0$.

The frame $\phi$ is adapted in the sense that $(g,f_1,f_2,\overline{f_1},\overline{f_2})$ together with Lie bracket induced by $\mathcal{L}^1$ define a complex $5$-dimensional Heisenberg algebra $\mathbb{C}\mathfrak{g}_{-}(p)$ that we call nondegenerate part of the bigraded symbol of $M_0$ at $p$ and that $e$ spans the Levi kernel. Indeed, the decomposition $(g),(f_1,f_2),(\overline{f_1},\overline{f_2})$ defines a natural bigrading
\[\mathbb{C}\mathfrak{g}_{-}(p)=\mathfrak{g}_{-2,0}(p)\oplus \mathfrak{g}_{-1,1}(p)\oplus\mathfrak{g}_{-1,-1}(p), \]
that is $[\mathfrak{g}_{i,j}(p),\mathfrak{g}_{k,l}(p)]\subset \mathfrak{g}_{i+k,j+l}(p).$ We denote by $\mathfrak{csp}(\mathbb{C}\mathfrak{g}_{-1}(p))$ the Lie algebra of derivations preserving the first grading of $\mathbb{C}\mathfrak{g}_{-}(p)$. Let us note that the first grading of $\mathbb{C}\mathfrak{g}_{-}(p)$ is actually the natural grading induced by the weights \eqref{grading convention} and there is the corresponding grading element in $\mathfrak{csp}(\mathbb{C}\mathfrak{g}_{-1}(p))$ that is an infinitesimal symmetry of $M_0$.

There is a natural bigrading 
\[\mathfrak{csp}(\mathbb{C}\mathfrak{g}_{-1}(p))=\mathfrak{csp}(\mathbb{C}\mathfrak{g}_{-1}(p))_{0,-2}\oplus \mathfrak{csp}(\mathbb{C}\mathfrak{g}_{-1}(p))_{0,0}\oplus \mathfrak{csp}(\mathbb{C}\mathfrak{g}_{-1}(p))_{0,2},\]
where $\mathfrak{csp}(\mathbb{C}\mathfrak{g}_{-1}(p))_{0,0}$ is spanned by the grading element and the linear transformations $L\in \mathfrak{gl}(2,\mathbb{C})$ in the basis $(f_1,f_2)$, $\mathfrak{csp}(\mathbb{C}\mathfrak{g}_{-1}(p))_{0,2}$ is spanned by $2\times 2$ symmetric matrices of the form $S^{0,2}=\Xi H(\zeta,\overline{\zeta})^{-1}$ for $\Xi$ mapping $(\overline{f_1},\overline{f_2})$ onto $(f_1,f_2)$ and $\mathfrak{csp}(\mathbb{C}\mathfrak{g}_{-1}(p))_{0,-2}$ is spanned by $2\times 2$ symmetric matrices of the form $S^{0,-2}=H(\zeta,\overline{\zeta})\overline{\Xi}$ for $\overline{\Xi}$ mapping $(f_1,f_2)$ onto $(\overline{f_1},\overline{f_2})$. We adopt this notation, because the $S^{0,\pm2}$ matrices provide explicit parametrization and insight into the geometric structure. The price we pay is that the complex conjugation induced on $\mathfrak{csp}(\mathbb{C}\mathfrak{g}_{-1}(p))$ takes the following form $\sigma_{\phi}$ in the frame $\phi:$
\begin{align}
    \sigma_{\phi}(L)&=-(H(\zeta,\overline{\zeta})^T)^{-1}\overline{L}^TH(\zeta,\overline{\zeta})^T\in \mathfrak{csp}(\mathbb{C}\mathfrak{g}_{-1}(p))_{0,0}\\
    \sigma_{\phi}(S^{0,2})&= H(\zeta,\overline{\zeta})\overline{S^{0,2}}H(\zeta,\overline{\zeta})^T\in \mathfrak{csp}(\mathbb{C}\mathfrak{g}_{-1}(p))_{0,-2}\\
    \sigma_{\phi}(S^{0,-2})&=  (H(\zeta,\overline{\zeta})^T)^{-1}\overline{S^{0,-2}}H(\zeta,\overline{\zeta}\big)^{-1}\in \mathfrak{csp}(\mathbb{C}\mathfrak{g}_{-1}(p))_{0,2}.
\end{align}

Now, let us compute the Lie brackets
\begin{align}
[e,\overline{f_k}](p)\equiv \sum_{l=1}^{2}\Xi(\phi(p))_{l,k} f_l\pmod{\mathcal{K}\oplus\overline{\mathcal{H}}}
\quad\quad\forall 1\leq k\leq 2.
\end{align}
At a point $p\in M_0$ with $\Im(w)=0,z_1=0,z_2=0$, we obtain
\begin{align}\label{key example Xi mat}
\Xi(\phi(p))=(H(\zeta,\overline{\zeta})^T)^{-1}S_{\zeta}(\zeta,\overline{\zeta}),\quad S^{0,2}(\phi(p))=(H(\zeta,\overline{\zeta})^T)^{-1}S_{\zeta}(\zeta,\overline{\zeta})H(\zeta,\overline{\zeta})^{-1}
\end{align}
and in particular,
\begin{align}
S^{0,2}(\phi(0))=\mathbf{S}_{\zeta}(0).
\end{align}

In other words, the Levi kernels $\mathcal{K}_p$ are embedded into $\mathfrak{csp}(\mathbb{C}\mathfrak{g}_{-1}(p))$ via the embedding $\iota:\mathcal{K}_p\to \mathfrak{csp}(\mathbb{C}\mathfrak{g}_{-1}(p))_{0,2}$ given by
\begin{align}\label{L kernel csp embedding}
\iota\big(e(p)\big):=S^{0,2}(\phi(p)).
\end{align}
It is shown in \cite{GKS2024} that $\iota$ -- although defined as an operator represented by a matrix w.r.t. a basis that are both given in terms of $\phi$  -- is indeed well defined by \eqref{L kernel csp embedding} independent from the choice of adapted frame $\phi$. Moreover, the CR manifold $M_0$ is $2$-nondegenerate at $p$ if $\iota(e)\neq 0$ and thus $M_0$ is uniformly $2$-nondegenerate around $0$ because we assume $\mathbf{S}_{\zeta}(0)\neq 0.$

Extending $\iota$ to an embedding $\iota:(\mathcal{K})_p\oplus (\overline{\mathcal{K}})_p\to \mathfrak{csp}(\mathbb{C}\mathfrak{g}_{-1}(p))_{0,2}\oplus \mathfrak{csp}(\mathbb{C}\mathfrak{g}_{-1}(p))_{0,-2}$ by the complex conjugation allows us to define the bigraded symbol.

\begin{definition}[introduced in \cite{porter2021absolute}]\label{CR symbol}
The \emph{bigraded symbol} $\mathbb{C}\mathfrak{g}_{\leq 0}(p)$ of $M_0$ at $p$ is the bigraded vector space 
\begin{align}\label{symbol decomposition}
\mathbb{C}\mathfrak{g}_{\leq 0}(p)&:=\mathfrak{g}_{-2,0}(p)\oplus\mathfrak{g}_{-1,-1}(p)\oplus\mathfrak{g}_{-1,1}(p)\oplus \mathfrak{g}_{0,-2}(p)\oplus \mathfrak{g}_{0,0}(p)\oplus \mathfrak{g}_{0,2}(p)\\&=\mathbb{C}\mathfrak{g}_{-}(p)\oplus \mathfrak{g}_{0,-2}(p)\oplus \mathfrak{g}_{0,0}(p)\oplus \mathfrak{g}_{0,2}(p)
\end{align}
given by $\mathfrak{g}_{0,2}(p):=\iota(\mathcal{K})_p$, $\mathfrak{g}_{0,-2}(p):=\iota(\overline{\mathcal{K}})_p$, and 
\[
\mathfrak{g}_{0,0}(p):=\{v\in \mathfrak{csp}(\mathbb{C}\mathfrak{g}_{-1}(p))\,|\, [v,w]\subset \mathfrak{g}_{j,k}(p)\,\forall\, w\in\mathfrak{g}_{j,k}(p),\,\forall (j,k)\in\mathcal{I} \},
\]
with $\mathcal{I}:=\{(-1,-1),(-1,1),(0,-2),(0,2)\}$, together with the antilinear involution given by restricting $v\mapsto \overline{v}$ from $\mathbb{C}\mathfrak{g}_{-}(p)\oplus \mathfrak{csp}(\mathbb{C}\mathfrak{g}_{-1}(p))$ to $\mathbb{C}\mathfrak{g}_{\leq 0}(p)$.
\end{definition}

So 
\begin{align}\label{first order constancy criterion alt a}
\iota(\mathcal{K})&\cong\mathrm{span}\left\{(H(\zeta,\overline{\zeta})^T)^{-1}S_{\zeta}(\zeta,\overline{\zeta})H(\zeta,\overline{\zeta})^{-1}\right\}
,\
\iota(\overline{\mathcal{K}})\cong\mathrm{span}\left\{\overline{S_{\zeta}(\zeta,\overline{\zeta})}\right\},\\
\iota(\mathcal{K}_0)&\cong\mathrm{span}\left\{\mathbf{S}_{\zeta}(0)\right\}
,\
\iota(\overline{\mathcal{K}}_0)\cong\mathrm{span}\left\{\mathbf{H}\overline{\mathbf{S}_{\zeta}(0)}\mathbf{H}^T\right\},
\end{align}
and one can conclude the following:

\begin{corollary}\label{matrix representation prop}
At $0\in M_0$, the bigraded symbol of the $2$-nondegenerate model $M_0$ given by \eqref{gen def fun}, \eqref{H from S2 formula}, and \eqref{S from S2 formula} is completely determined by $\mathbf{H}$ and $\mathbf{S}_{\zeta}(0)$.
\end{corollary}

Let us continue by computing the Lie brackets
\begin{align}
[e,f_k](p)\equiv\sum_{l=1}^{2}(\Omega\big(\phi\big)(p))_{l,k} f_l\pmod{\mathcal{K}\oplus\overline{\mathcal{K}}}
\quad\quad\forall 1\leq k\leq 2.
\end{align}
At a point $p\in M_0$ with $\Im(w)=0,z_1=0,z_2=0$, we obtain
\begin{align}\label{key example Omega mat}
\Omega\big(\phi\big)=(H(\zeta,\overline{\zeta})^T)^{-1}H_{\zeta}(\zeta,\overline{\zeta})^T
\end{align}
and in particular,
\begin{align}
\Omega\big(\phi\big)(0)=0.
\end{align}

In other words, the Levi kernels $\mathcal{K}_p$ are embedded into $\mathfrak{csp}(\mathbb{C}\mathfrak{g}_{-1}(p))$ via the embedding $\iota_\phi:\mathcal{K}_p\to \mathfrak{csp}(\mathbb{C}\mathfrak{g}_{-1}(p))_{0,2}\oplus  \mathfrak{csp}(\mathbb{C}\mathfrak{g}_{-1}(p))_{0,0}$ given by
\begin{align}\label{iota phi map}
\iota_\phi\big(e(p)\big):=\left(S^{0,2}(\phi(p)),\Omega\big(\phi\big)(p)\right).
\end{align}
It is shown in \cite{GKS2024} that $\iota_\phi$ is indeed well defined by \eqref{iota phi map}, but depends on the adapted frame $\phi$ in a neighborhood of $p$ and not just at $p$. This provides the following modified version of the bigraded symbol, where $\iota_\phi: \overline{\mathcal{K}}_p\to \mathfrak{csp}(\mathbb{C}\mathfrak{g}_{-1}(p))_{0,-2}\oplus  \mathfrak{csp}(\mathbb{C}\mathfrak{g}_{-1}(p))_{0,0}$ is again obtained by the complex conjugation $\sigma_{\phi}$.

\begin{definition}[{\cite[Definition 2.16]{GKS2024}}]\label{mod symb in frame} For $p\in M_0$ and adapted frame $\phi$, the \emph{modified symbol $\mathfrak{g}_{\leq 0}^{\mathrm{mod}}\big(\phi;p\big)$ at the point $p$ along $\phi$} is the graded vector space
\[
\mathfrak{g}_{\leq 0}^{\mathrm{mod}}\big(\phi;p\big):=\mathbb{C}\mathfrak{g}_{-}(p)\oplus \mathfrak{g}_0^{\mathrm{mod}}\big(\phi;p\big)=\mathbb{C}\mathfrak{g}_{-}(p)\oplus\mathfrak{g}_{0,-}^{\mathrm{mod}}(\phi; p\big)\oplus \mathfrak{g}_{0,0}(p)\oplus\mathfrak{g}_{0,+}^{\mathrm{mod}}(\phi; p\big),
\]
given by
\[
\mathfrak{g}_0^{\mathrm{mod}}\big(\phi;p\big):=\iota_{\phi}\left(\overline{\mathcal{K}}(p)\right)+\mathfrak{g}_{0,0}(p)+\iota_{\phi}\left(\mathcal{K}(p)\right).
\]
\end{definition}
Notice that the modified symbol at $0$ along $\phi$ is determined by the three matrices $\mathbf{H}$, $S^{0,2}(\phi(0))$, and $\Omega\big(\phi\big)(0)$.

We would like to remove the dependence on the choice of the adapted frame from this definition. To do so let us compute the change of the bigraded symbol a points $p\in M_0$ with $\Im(w)=0,z_1=0,z_2=0$:
\begin{align}
e(\iota(\overline{e}))(p)&=
\overline{S_{\zeta\overline{\zeta}}(\zeta,\overline{\zeta})}=H_{\zeta}(\zeta,\overline{\zeta})H(\zeta,\overline{\zeta})^{-1}\overline{S_{\zeta}(\zeta,\overline{\zeta})}+\overline{S_{\zeta}(\zeta,\overline{\zeta})}(H(\zeta,\overline{\zeta})^{-1})^TH_{\zeta}(\zeta,\overline{\zeta})^T,\\
e(\iota(e))(p)=&-(H(\zeta,\overline{\zeta})^T)^{-1}H_{\zeta}(\zeta,\overline{\zeta})^T(H(\zeta,\overline{\zeta})^T)^{-1}S_{\zeta}(\zeta,\overline{\zeta})H(\zeta,\overline{\zeta})^{-1}+\\
&+(H(\zeta,\overline{\zeta})^T)^{-1}S_{\zeta\zeta}(\zeta,\overline{\zeta})H(\zeta,\overline{\zeta})^{-1}+\\
&-(H(\zeta,\overline{\zeta})^T)^{-1}S_{\zeta}(\zeta,\overline{\zeta})H(\zeta,\overline{\zeta})^{-1}H_{\zeta}(\zeta,\overline{\zeta})H(\zeta,\overline{\zeta})^{-1}.
\end{align}
and in particular
\begin{align}
e(\iota(\overline{e}))(0)&=0,\ e(\iota(e))(0)=\mathbf{S}_{\zeta\zeta}(0).
\end{align}
In other words, we define bilinear maps
\[O_{0,-2}(\phi)(p):\mathcal{K}_p\otimes \mathfrak{g}_{0,-2}(p)\to \mathfrak{csp}(\mathbb{C}\mathfrak{g}_{-1}(p))_{0,-2}/\mathfrak{g}_{0,-2}(p)\]
and 
\[O_{0,2}(\phi)(p):\mathcal{K}_p\otimes \mathfrak{g}_{0,2}(p)\to \mathfrak{csp}(\mathbb{C}\mathfrak{g}_{-1}(p))_{0,2}/\mathfrak{g}_{0,2}(p)\]
by \[O_{0,-2}(\phi)(p)(X,S^{0,-2}(\phi(p))):=X\left(S^{0,-2}(\phi)\right)(p)\] and \[O_{0,2}(\phi)(p)(X,S^{0,2}(\phi(p))):=X\left(S^{0,2}(\phi)\right)(p).\] It is shown in \cite{GKS2024} that $O_{0,-2}(\phi),O_{0,2}(\phi)$ are indeed well defined, but depends on the adapted frame $\phi$. What is crucial is that the dependence of $\Omega(\phi)(p)$ and $O_{0,\pm2}(\phi)(p)$ on $\phi$ is connected in the following way:
\[
\Omega(\phi^\prime)(p)=\Omega(\phi)(p)+B, B\in \mathfrak{gl}(2,\mathbb{C})
\]
if and only if
\begin{align}\label{equiv class O}
O_{0,-2}(\phi^\prime)(p)\left(e,S^{0,-2}(\phi^\prime(p))\right)&=O_{0,-2}(\phi)(p)\left(e,S^{0,-2}(\phi(p))\right)\\
&\quad+B^TS^{0,-2}(\phi(p))+S^{0,-2}(\phi(p))B, \\
O_{0,2}(\phi^\prime)(p)\left(e,S^{0,2}(\phi^\prime(p))\right)&=O_{0,2}(\phi)(p)\left(e,S^{0,2}(\phi(p))\right)\\
&\quad-BS^{0,2}(\phi(p))-S^{0,2}(\phi(p))B^T.
\end{align}
Therefore, posing a normalization condition on the equivalence classes \eqref{equiv class O} by choosing a unique pair of bilinear maps 
\[
\big(O_{0,-2}^N(\phi(p)),O_{0,2}^N(\phi(p))\big)\in \left\{\left.\left(O_{0,-2}(\phi^\prime)(p),O_{0,2}(\phi^\prime)(p)\right)\,\right|\, \phi^\prime(p)=\phi(p)\right\}
\]
in each equivalence class provides the following notion of modified symbols.

\begin{definition}\label{normalized sections}
We say that an adapted frame $\phi$ is \emph{normalized with respect to a choice of normalization condition $\big(O_{0,-2}^N(\phi(p)),O_{0,2}^N(\phi(p))\big)$ at $p$} if
$O_{0,\pm 2}(\phi)(p)=O_{0,\pm2}^N(\phi(p))$. We say that $\mathfrak{g}_{\leq0}^{\mathrm{mod}}(\tilde \phi(p)):=\mathfrak{g}_{\leq0}^{\mathrm{mod}}({ \phi;p})$ is the \emph{modified symbol at $\phi(p)$ with respect to normalization condition $\big(O_{0,-2}^N(\phi(p)),\allowbreak O_{0,2}^N(\phi(p))\big)$} if $\phi$ is normalized with respect to $\big(O_{0,-2}^N(\phi(p)),\allowbreak O_{0,2}^N(\phi(p))\big)$. 

We say that $(M_0,\mathcal{H})$ has \emph{constant bigraded symbol up to first order at $p$} if $O_{0,\pm2}^N(\phi(p))=0$ for any adapted frame $\phi$ around $p$, and otherwise we say that $O_{0,\pm2}^N(\phi(p))$ are obstructions to first order constancy at $p$.
\end{definition}

So in order to obtain the modified symbol and obstruction to first order constancy at $0$ we need to choose normalization condition $(O_{0,-2}^N(\phi(0)),O_{0,2}^N(\phi(0)))$. We choose $O_{0,-2}^N(\phi(0))=0$, which reduces freedom of $\Omega(\phi(0))$ to $N_{\mathfrak{gl}(2,\mathbb{C})}(\mathfrak{g}_{0,-2}(0))$ that consists of $B$ such that $B^TS^{0,-2}(\phi(0))+S^{0,-2}(\phi(0))B\in \mathrm{span}\{S^{0,-2}(\phi(0))\}$, i.e., 
\[N_{\mathfrak{gl}(2,\mathbb{C})}(\mathfrak{g}_{0,-2}(0)):=
\left\{B\in \mathfrak{gl}(2,\mathbb{C})\,\left|\,B^T\mathbf{H}\overline{\mathbf{S}_{\zeta}(0)}\mathbf{H}^T+\mathbf{H}\overline{\mathbf{S}_{\zeta}(0)}\mathbf{H}^TB\in \mathrm{span}\{\mathbf{H}\overline{\mathbf{S}_{\zeta}(0)}\mathbf{H}^T\}
\right.\right\}.
\]
Therefore, to fix $O_{0,2}^N(\phi(0))$ we need to decompose
\[\mathfrak{csp}(\mathbb{C}\mathfrak{g}_{-1}(0))_{0,2}/\mathfrak{g}_{0,2}(0)=N_{0,2}\oplus \left\{B\mathfrak{g}_{0,2}(0)+\mathfrak{g}_{0,2}(0)B^T:\ B\in N_{\mathfrak{gl}(2,\mathbb{C})}(\mathfrak{g}_{0,-2}(0))\right\}.\]
If we denote $\pi_{N_{0,2}}$ the projection to $N_{0,2}$ in this decomposition, then 
\begin{align}\label{rank 1 normalization}
O^N_{0,2}( \phi(0))&:=\pi_{N_{0,2}}\circ O_{0,2}(\phi)(0),\\
O^N_{0,2}( \phi(0))(e,\mathbf{S}_{\zeta}(0))&=\pi_{N_{0,2}}(\mathbf{S}_{\zeta\zeta}(0))=\mathbf{S}_{\zeta\zeta}(0)-\Omega(\phi(0))\mathbf{S}_\zeta(0)-\mathbf{S}_\zeta(0)\Omega(\phi(0))^T,\\ \Omega(\phi(0))&\in N_{\mathfrak{gl}(2,\mathbb{C})}(\mathfrak{g}_{0,-2}(0))
\end{align}
defines $\Omega(\phi(0))$ up to $\mathfrak{g}_{0,0}(0)$ and we can conclude the following.

\begin{corollary}\label{mod symbol determination corollary}
At $0\in M_0$, the modified symbol and the obstructions to first order constancy for a fixed normalization condition $N_{0,2}$ of the $2$-nondegenerate model $M_0$ given by \eqref{gen def fun}, \eqref{H from S2 formula}, and \eqref{S from S2 formula} are completely determined by $\mathbf{H}$, $\mathbf{S}_{\zeta}(0)$ and $\mathbf{S}_{\zeta\zeta}(0)$. With respect to $N_{0,2}$ the modified symbol itself is represented by $\mathbf{H}$, $\mathbf{S}_{\zeta}(0)$, and $\Omega(\phi(0))$.
\end{corollary}

Matrices $\mathbf{H}$ and $\mathbf{S}_\zeta(0)$ represent the bigraded symbols as in Corollary \ref{matrix representation prop}, and all their possible bigraded symbol equivalence class representatives can be obtained \cite[formulas (17)-(24)]{SykesHomogeneous} (we incorporate (21) as a special case of (22)). However, it turns out that to obtain nicer normal forms, we need to change the convention for \cite[f. (18) and (20)]{SykesHomogeneous} and obtain the third and fifth row of Table \ref{normalization conditions Table}. Note that \cite[f. (18)]{SykesHomogeneous}, which has two formulas depending on $
\epsilon$, is transformed by the matrix $\left(\begin{smallmatrix} \frac{1}{\sqrt{2}}& \frac{1}{\sqrt{2}}\\ \frac{i}{\sqrt{2}}& \frac{-i}{\sqrt{2}}\end{smallmatrix}\right)$ to a formula in the third row in the $\epsilon=1$ case and the fifth row in the $\epsilon=-1$ case. Analogously, \cite[f. (20)]{SykesHomogeneous} is transformed by the matrix $\left(\begin{smallmatrix} \frac{-1-i}{2}& \frac{-1+i}{2}\\ \frac{-1-i}{2}& \frac{1-i}{2}\end{smallmatrix}\right)$ to a formula in the third row.

Let us fix the normalization conditions for these representatives.

\begin{proposition}\label{table2prop}
    The spaces $N_{0,2}$ comprised of matrices in the image of $\pi_{N_{0,2}}$ from Table \ref{normalization conditions Table} provide normalization conditions for all equivalence class representatives $\mathbf{H}$ and $S^{0,2}=\mathbf{S}_\zeta(0)$ of bigraded symbols of uniformly $2$-nondegenerate CR hypersurfaces in $\mathbb{C}^4.$
\end{proposition}
\begin{table}
    \centering
    \begin{tabular}{|c|c|c|c|}\hline
     \To{7}\Bo{6} $\mathbf{H}$ & $S^{0,2}$& \parbox{3.5cm}{\centering $\pi_{N_{0,2}}$ applied to\\ $\left(\begin{matrix}
          f_1&f_2\\f_2&f_3
      \end{matrix}\right)$} & $N_{\mathfrak{gl}(2,\mathbb{C})}(\mathfrak{g}_{0,-2}(0))$ \\\hline
     \To{5}\Bo{4} $\left(\begin{matrix}
          1&0\\0&\epsilon
      \end{matrix}\right)$
      &$\left(\begin{matrix}
          1&0\\0&\lambda
      \end{matrix}\right)$&$\left(\begin{matrix}
           0&0\\ 0&f_3-\lambda f_1
      \end{matrix}\right)$& $\left(\begin{matrix}
          a&-\lambda b\\b&a
      \end{matrix}\right)$\\\hline
     \To{5}\Bo{4} 
     $\left(\begin{matrix}
          0&1\\1&0
      \end{matrix}\right)$
      &$\left(\begin{matrix}
          1&0\\0&e^{i\theta}
      \end{matrix}\right)$ &$\left(\begin{matrix}
           0&0\\ 0&f_3-e^{i\theta}f_1
      \end{matrix}\right)$& $\left(\begin{matrix}
          a&-e^{i\theta} b\\b&a
      \end{matrix}\right)$\\\hline
     \To{5}\Bo{4} 
     $\left(\begin{matrix}
          1&0\\0&\epsilon
      \end{matrix}\right)$
      &$\left(\begin{matrix}
          0&1\\1&0
      \end{matrix}\right)$ &$\left(\begin{matrix}
          f_1&0\\ 0&f_3
      \end{matrix}\right)$&$\left(\begin{matrix}
          a&0\\ 0&b
      \end{matrix}\right)$\\\hline
     \To{5}\Bo{4} 
      $\left(\begin{matrix}
          1&0\\0&\epsilon
      \end{matrix}\right)$
      &$\left(\begin{matrix}
          1&0\\0&0
      \end{matrix}\right)$ &$\left(\begin{matrix}
          0&0\\ 0&f_3
      \end{matrix}\right)$&$\left(\begin{matrix}
          a&0\\ c&b
      \end{matrix}\right)$\\\hline
     \To{5}\Bo{4} 
      $\left(\begin{matrix}
          0&1\\1&0
      \end{matrix}\right)$
      &$\left(\begin{matrix}
          0&1\\1&0
      \end{matrix}\right)$ &$\left(\begin{matrix}
          f_1&0\\ 0&f_3
      \end{matrix}\right)$&$\left(\begin{matrix}
          a&0\\ 0&b
      \end{matrix}\right)$\\\hline
      \To{5}\Bo{4} 
      $\left(\begin{matrix}
          0&1\\1&0
      \end{matrix}\right)$
      &$\left(\begin{matrix}
          1&1\\1&0
      \end{matrix}\right)$ &$\left(\begin{matrix}
          0&0\\ 0&f_3
      \end{matrix}\right)$&$\left(\begin{matrix}
          a&b\\ 0&a-2b
      \end{matrix}\right)$\\\hline
      \To{5}\Bo{4} 
      $\left(\begin{matrix}
          0&1\\1&0
      \end{matrix}\right)$
      &$\left(\begin{matrix}
          1&0\\0&0
      \end{matrix}\right)$ &$\left(\begin{matrix}
          0&f_2\\ f_2&f_3
      \end{matrix}\right)$&$\left(\begin{matrix}
          a&b\\ 0&c
      \end{matrix}\right)$\\\hline
     \end{tabular}
    \caption{Normalization conditions $N_{0,2}$ corresponding to bigraded symbols of CR hypersurfaces in $\mathbb{C}^4$. Parameters $a,b,c$ are complex, $\lambda,\theta$ are real, $\lambda> 1$,$0<\theta< \pi$, and $\epsilon=\pm1$.}
    \label{normalization conditions Table}
\end{table}
\begin{proof}
    For each case in Table \ref{normalization conditions Table}, we can straightforwardly compute $N_{\mathfrak{gl}(2,\mathbb{C})}(\mathfrak{g}_{0,-2}(0))\allowbreak=N_{\mathfrak{gl}(2,\mathbb{C})}(\mathrm{span}\{(\mathbf{H}^T)^{-1}\overline{S^{0,2}}\mathbf{H}^{-1}\})$ for the given representatives of the equivalence classes of bigraded symbol. Then we can check that claimed $N_{0,2}$ provides complement to \[\mathrm{span}\left\{BS^{0,2}+S^{0,2}B^T:\ B\in N_{\mathfrak{gl}(2,\mathbb{C})}(\mathrm{span}\{(\mathbf{H}^T)^{-1}\overline{S^{0,2}}\mathbf{H}^{-1}\}\right\}\]
    in the space of symmetric $2\times 2$ matrices.

    As for the realizibility of the bigraded symbols, it suffices to consider $2$-nondegenerate model with $\mathbf{S}(\zeta)=\zeta S^{0,2},$ for which we computed that the bigraded symbol is represented precisely by $\mathbf{H}$ and $\mathbf{S}_\zeta(0)=S^{0,2}.$
\end{proof}

In \cite{GKS2024}, we show that the $2$-nondegenerate models with
\begin{align}\label{new S2}
(\mathbf{S}(\zeta),0)=\ln\left(\exp\left((S^{0,2}\zeta,\Omega\zeta)\right)\exp\left((0,-\Omega\zeta)\right)\right)=
\left(\left(\begin{matrix}f_1&f_2\\f_2 &f_3\end{matrix}\right),0
\right)
\end{align}
where $\left(S^{0,2}\zeta,\Omega\zeta\right),(0,-\Omega\zeta)\in \mathfrak{csp}(\mathbb{C}\mathfrak{g}_{-1}(0))_{0,2}\oplus\mathfrak{csp}(\mathbb{C}\mathfrak{g}_{-1}(0))_{0,0}$ realize the modified symbols with $\Omega\in N_{\mathfrak{gl}(2,\mathbb{C})}(\mathrm{span}\{(\mathbf{H}^T)^{-1}\overline{S^{0,2}}\mathbf{H}^{-1}\})$ and bigraded symbol represented by $\mathbf{H}$ and $S^{0,2}$. Indeed,
\begin{align}
S_{\zeta\zeta}(0,0)&=\mathbf{H}^T\mathbf{S}_{\zeta\zeta}(0)\mathbf{H}=\frac12\mathbf{H}^T\left(\Omega S^{0,2}+S^{0,2}\Omega^T\right)\mathbf{H},
\end{align}
follows from the expansion 
\begin{align}\label{S2 formula}
\mathbf{S}(\zeta)=S^{0,2}\zeta+(\Omega S^{0,2}+S^{0,2}\Omega^T)\zeta^2+O(\zeta^3)
\end{align}
derived using the Baker-Campbell-Hausdorff formula. Taking $B=\Omega$ in \eqref{equiv class O} thus ensures $O^N_{0,-2}(\phi(0))=O^N_{0,2}(\phi(0))=0$ and we have a realization that is constant up to first order.

This together with Table \ref{normalization conditions Table} allows us to classify all possible modified symbols of uniformly $2$-nondegenerate hypersurfaces in $\mathbb{C}^4$ and determine the intersection $\mathrm{CU}(\mathbf{H})\cap G_{0,0}$ of
\begin{align}\label{CUdef}
 \mathrm{CU}(\mathbf{H})=
\left\{(u,U)\in \mathbb{C}^*\times \mathrm{GL}(2,\mathbb{C})/(\mathbb{Z}_2=\{(1,\mathrm{Id}),(-1,-\mathrm{Id})\}) \,\left|\, u\overline{u}^{-1}U^T\mathbf{H}\overline{U}=U\right.\right\}   \end{align}
and 
\begin{align}\label{G00def}
G_{0,0}=
\left\{(u,U)\in \mathbb{C}^*\times \mathrm{GL}(2,\mathbb{C})/\mathbb{Z}_2\,\left|\,\parbox{8.5cm}{$\mathrm{span}\{U^{-1}\mathbf{S}_\zeta(0)(U^{-1})^T
\}=\mathrm{span}\{\mathbf{S}_\zeta(0)\}$,\\
$\mathrm{span}\{U^T\mathbf{H}\overline{\mathbf{S}_\zeta(0)}\mathbf{H}^TU\}=\mathrm{span}\{\mathbf{H}\overline{\mathbf{S}_\zeta(0)}\mathbf{H}^T\}$} \right.\right\}.
\end{align}
\begin{proposition}\label{degree 0 symmetry bound}
    The entries of Table \ref{7-d. symbol table} represent all classes of modified  symbols of uniformly $2$-nondegenerate CR hypersurfaces in $\mathbb{C}^4$, where $\mathbf{H}$ and $S^{0,2}=\mathbf{S}_\zeta(0)$ represent the equivalence class of bigraded symbols and $\Omega=\Omega(\phi(0))$ is unique up to a multiple of the grading element and elements of the given $\mathfrak{g}_{0,0}^\prime=\mathfrak{g}_{0,0}\cap \mathfrak{gl}(2,\mathbb{C})$. Moreover, real forms given in the table represent the Lie algebra of $\mathrm{CU}(\mathbf{H})\cap G_{0,0}$.  The Lie algebra of the stabilizer of $\Omega$ in $\mathrm{CU}(\mathbf{H})\cap G_{0,0}$ is specified by the \emph{annihilator of $\Omega$} condition in Table \ref{7-d. symbol table}, and this subgroup acts by symmetries on the $2$-nondegenerate model given by \eqref{new S2}.
\end{proposition}
\begin{table}[htbp]
    \centering
    \begin{tabular}{|c|c|c|c|c|}\hline
      \multirow{2}*{$\mathbf{H}$} & \multirow{2}*{$S^{0,2}$}& \multirow{2}*{$\mathfrak{g}_{0,0}^\prime$ }& \Bo{2}$\Re(\mathfrak{g}_{0,0}^\prime)\subset\mathfrak{g}_{0,0}^\prime$  & \multirow{2}*{  realizable $\Omega$ } \\\cdashline{4-4}
      &&&annihilator of $\Omega$&\\\hline
      \To{3}
     \multirow{2}*{$\left(\begin{matrix}
          1&0\\0&\epsilon
      \end{matrix}\right)$}
      &\multirow{2}*{$\left(\begin{matrix}
          1&0\\0&\lambda
      \end{matrix}\right)$ } &\multirow{2}*{$\left\{\left(\begin{matrix}
          a&0\\0&a
      \end{matrix}\right)\right\}$}& 
      $a=-\overline{a}$
      &\multirow{2}*{$\left(\begin{matrix}
          0&-\tau\lambda \\ \tau&0
      \end{matrix}\right)$}\\\cdashline{4-4}\Bo{2} 
      &&&$a=0$&\\\hline
     \To{3}
     \multirow{2}*{
     $\left(\begin{matrix}
          0&1\\1&0
      \end{matrix}\right)$}
      &
      \multirow{2}*{$\left(\begin{matrix}
          1&0\\0&e^{i\theta}
          \end{matrix}\right)$}
      &
      \multirow{2}*{$\left\{\left(\begin{matrix}
          a&0\\0&a
      \end{matrix}\right)\right\}$}
      & $a=-\overline{a}$ &\multirow{2}*{$\left(\begin{matrix}
          0&-\tau e^{i\theta}\\ \tau&0
      \end{matrix}\right)$}\\\cdashline{4-4}\Bo{2} &&&$a=0$&\\\hline
     \To{5} \Bo{4}
     $\left(\begin{matrix}
          1&0\\0&\epsilon
      \end{matrix}\right)$
      &$\left(\begin{matrix}
          0&1\\1&0
      \end{matrix}\right)$ &$\left\{\left(\begin{matrix}
          a&0\\0&b
      \end{matrix}\right)\right\}$& $a=-\overline{a}$, $b=- \overline{b}$&$\left(\begin{matrix}
          0&0\\ 0&0
      \end{matrix}\right)$\\\hline
     \To{3} 
      \multirow{2}*{$\left(\begin{matrix}
          1&0\\0&\epsilon
      \end{matrix}\right)$}
      &\multirow{2}*{$\left(\begin{matrix}
          1&0\\0&0
      \end{matrix}\right)$ }&\multirow{2}*{$\left\{\left(\begin{matrix}
          a&0\\0&b
      \end{matrix}\right)\right\}$}& $a=-\overline{a},b=-\overline{b}$&\multirow{2}*{$\left(\begin{matrix}
          0&0\\ \tau&0
      \end{matrix}\right)$}\\\cdashline{4-4}\To{2}\Bo{2}&&&$b=3a$&\\\hline
     \To{5}\Bo{4} 
      $\left(\begin{matrix}
          0&1\\1&0
      \end{matrix}\right)$
      &$\left(\begin{matrix}
          0&1\\1&0
      \end{matrix}\right)$ &$\left\{\left(\begin{matrix}
          a&0\\0&b
      \end{matrix}\right)\right\}$&$a=-\overline{b}$&$\left(\begin{matrix}
          0&0\\ 0&0
      \end{matrix}\right)$\\\hline
      \To{3}
      \multirow{2}*{$\left(\begin{matrix}
          0&1\\1&0
      \end{matrix}\right)$}
      &\multirow{2}*{$\left(\begin{matrix}
          1&1\\1&0
      \end{matrix}\right)$} &\multirow{2}*{$\left\{\left(\begin{matrix}
          a&0\\0&a
      \end{matrix}\right)\right\}$}&$a=-\overline{a}$ &\multirow{2}*{$\left(\begin{matrix}
          2\tau&\tau\\ 0&0
      \end{matrix}\right)$}\\\cdashline{4-4}\Bo{2} &&& $a=0$&\\\hline
      \To{5}\Bo{4} 
      $\left(\begin{matrix}
          0&1\\1&0
      \end{matrix}\right)$
      &$\left(\begin{matrix}
          1&0\\0&0
      \end{matrix}\right)$ &$\left\{\left(\begin{matrix}
          a&b\\0&c
      \end{matrix}\right)\right\}$&$a=-\overline{c},b=-\overline{b}$&$\left(\begin{matrix}
          0&0\\ 0&0
      \end{matrix}\right)$\\\hline
    \end{tabular}
    \caption{Realizable modified symbols of $7$-dimensional $2$-nondegenerate hypersurfaces. Parameters $a,b,c$ are complex, $\lambda,\theta,\tau$ are real, $\lambda> 1$,$0<\theta< \pi$, $\tau\geq 0$, and $\epsilon=\pm1$.}
    \label{7-d. symbol table}
\end{table}
\begin{proof}
For the given representatives of the bigraded symbol, we straightforwardly compute $\mathfrak{g}_{0,0}^\prime\subset N_{\mathfrak{gl}(2,\mathbb{C})}(\mathfrak{g}_{0,-2}(0))$. By comparing $\mathfrak{g}_{0,0}^\prime$ with $N_{\mathfrak{gl}(2,\mathbb{C})}(\mathfrak{g}_{0,-2}(0))$ from Table \ref{normalization conditions Table}, we can pick representatives of $\Omega$. Since the action of $\mathrm{CU}(\mathbf{H})\cap G_{0,0}$ gets normalized to identity on $S^{0,2}$ by reparametrization $g_U$, we need to take it into account when computing the action on $\Omega$, which implies that we can always rescale $\Omega$ by complex unit number to make $\tau\geq 0.$  As the remainder of $\mathrm{CU}(\mathbf{H})\cap G_{0,0}$ preserves $\Omega$, it is a symmetry as a consequence of \cite[Proposition 4.3]{GKS2024}. This provides the claimed conditions.
\end{proof}
Thus we get the following corollary of Proposition \ref{degree 0 symmetry bound} when we take into account that the grading element is always a symmetry of a $2$-nondegenerate model and thus the Lie algebra of symmetries is weighted graded by its eigenspaces.

\begin{corollary}\label{C4 degree zero bound}
    For modified symbols from Table \ref{7-d. symbol table}, the $2$-nondegenerate models given by \eqref{new S2} have maximal dimension of weighted degree zero isotropy, among all other $2$-nondegenerate CR hypersurfaces with this modified symbol (at a point).
\end{corollary}

\section{A complete normal form for models in \texorpdfstring{$\mathbb{C}^{4}$}{C\^4}}\label{The equivalence problem for models}\label{A complete normal form}

Here we prove Theorem \ref{main theorem} and achieve a complete normal form for models in $\mathbb{C}^{4}$. Recall Theorem \cite[Theorem 6.5]{GKS2024} gives a partial normal form by fixing  $\mathbf{H}$, $\mathbf{S}_\zeta(0)$ and $\mathbf{S}_{\zeta\zeta}(0)$ (uniquely defined up to a group action of $\mathrm{CU}(\mathbf{H})\cap G_{0,0}$ that we know from Proposition \ref{degree 0 symmetry bound}) and fixing the first nonzero holomorphic function $\mathbf{S}(\zeta)_{(i,j)}=\zeta$ in the ordering $(1,1)<(1,2)<(2,2)$ for the positions $(i,j)$. More precisely, for $(u,U)\in \mathrm{CU}(\mathbf{H})\cap G_{0,0}$, there is a unique local biholomorphism $g_{(u,U)}:\mathbb{C}\to \mathbb{C}$ fixing $0$ such that for the change of coordinates
\begin{equation}
\label{symgrouptransform}
   (w,z,\zeta)\mapsto(u^{-2}w,(z_1,z_2)u^{-1}U^T,\zeta+g_{u,U}(\zeta))
\quad\quad\forall\, (u,U)\in \mathrm{CU}(\mathbf{H})\cap G_{0,0}, 
\end{equation}
the transformed
    \begin{align}
        \tilde{\mathbf{H}}&=u\overline{u}^{-1}U^T\mathbf{H}\overline{U},\\
        \tilde{\mathbf{S}}(\zeta)&=U^{-1}\mathbf{S}(g_{(u,U)}(\zeta))(U^{-1})^T
    \end{align}
are still in the partial normal form, i.e., $\tilde{\mathbf{S}}(\zeta)_{(i,j)}=\zeta$ holds for the first nonzero position $(i,j)$. To derive a complete normal form for $2$-nondegenerate models in $\mathbb{C}^4$, we need to apply the action of $\mathrm{CU}(\mathbf{H})\cap G_{0,0}$ to normalize terms in the series expansion of $\mathbf{S}(\zeta)$ in a prescribed way. We do this until the subgroup of $\mathrm{CU}(\mathbf{H})\cap G_{0,0}$ preserving these additional normalizations consists only of symmetries of the normalized defining equation. The normal subgroup $\{(u,\pm u\mathrm{Id})\,|\,u>0\}$ acts on every model by symmetries, so for simplicity we may reduce to finding subgroups in the quotient group
\begin{align}\label{CU quotient}
\left(\mathrm{CU}(\mathbf{H})\cap G_{0,0}\right)/\{(u,\pm u\mathrm{Id})\,|\,u>0\}\cong\{(u,U)\in G_{0,0}\,|\,u=1\mbox{ or }u= i\}/\{(1,\pm\mathrm{Id})\}.
\end{align}
Therefore, we  can separate our results according to dimension of symmetries that remain in \eqref{CU quotient} when the process completes. We first list our results in the following $4$ propositions that will be proven in Subsections \ref{row cases subsection a}, \ref{row cases subsection b}, \ref{row cases subsection c}, \ref{row cases subsection d}, and \ref{row cases subsection e}.

Let us emphasize that the term \emph{complete normal} form as it is usually used in the literature does not mean that the action of $\mathrm{CU}(\mathbf{H})\cap G_{0,0}$ is completely exhausted, but that all free parameters in the normal form are identified that can be, subject to further conditions depending on an action of a discrete subgroup $\mathrm{CU}(\mathbf{H})\cap G_{0,0}$. Consequently, it is reasonable to present the subalgebra of $\mathfrak{u}(\mathbf{H})\cap \mathfrak{g}_{0,0}$ consisting of infinitesimal symmetries that in our cases can take dimension $3$ (i.e., maximal), $2$, $1$, and $0$. 
Note that these infinitesimal symmetries do not comprise a model's full isotropy subalgebra, as, in particular, the grading element provides an additional dimension of symmetries for every model and positively weighted isotropy can also appear. The discrete action appearing for our complete normal forms is characterized by the following remark.

\begin{remark}\label{discrete normalization remark}
    We prescribe the complete normal form only up to certain discrete group actions, because specifying it further entails fixing ad hoc choices that readers can easily make, perhaps even fixing these choices in a manner informed by some applications. In \S\ref{row cases subsection a}-\ref{row cases subsection e} we describe how these discrete groups act on the terms in the expansion of $\mathbf{S}(\zeta)$ that one normalizes with the action. Generally, these actions have the form $c\zeta^\mu\mapsto e^{i\theta \pi}c\zeta^\mu$ for some $\theta\in\mathbb{Q}$, generating a finite set of possible values for how the coefficient of $c\zeta^\mu$ is transformed.
\end{remark}

\begin{proposition}\label{table4prop3}
    A $2$-nondegenerate model in $\mathbb{C}^4$ with $3$-dimensional subalgebra of $\mathfrak{u}(\mathbf{H})\cap \mathfrak{g}_{0,0}$ consisting of infinitesimal symmetries is locally biholomorphicaly equivalent to the $2$-nondegenerate model with $\mathbf{H}=\left(\begin{matrix}
          0&1\\1&0
      \end{matrix}\right)$ and $\mathbf{S}(\zeta)= \left(\begin{matrix}
          \zeta&0\\0&0
      \end{matrix}\right)$, which is Type \Rmnum{7} in Table \ref{table of defining equations}. The subalgebra of $\mathfrak{u}(\mathbf{H})\cap \mathfrak{g}_{0,0}$ consisting of infinitesimal symmetries is
      \[
\left\{\left.
\left(
\begin{array}{cc}
-\theta_2+i\theta_1 & \theta_3 \\
0 & \theta_2+i\theta_1
\end{array}
\right)
\,\right|\,
\theta_j\in \mathbb{R}
\right\}.
\]
\end{proposition}

\begin{proposition}\label{table4prop2}
    A $2$-nondegenerate model in $\mathbb{C}^4$ with 2-dimensional subalgebra of $\mathfrak{u}(\mathbf{H})\cap \mathfrak{g}_{0,0}$ consisting of infinitesimal symmetries is locally biholomorphicaly equivalent to one of the following $2$-nondegenerate models:
    \begin{enumerate}
      \item The model with $\mathbf{H}=\left(\begin{matrix}
          1&0\\0&-1
      \end{matrix}\right)$ and $\mathbf{S}(\zeta)= \left(\begin{matrix}
          0&\zeta\\\zeta&0
      \end{matrix}\right)$, which is Type \Rmnum{6} in Table \ref{table of defining equations}. The subalgebra of $\mathfrak{u}(\mathbf{H})\cap \mathfrak{g}_{0,0}$ consisting of infinitesimal symmetries is
      \[
\left\{\left.
\left(
\begin{array}{cc}
i\theta_1 & 0 \\
0 & i\theta_2
\end{array}
\right)
\,\right|\,
\theta_j\in \mathbb{R}
\right\}.
\]
      \item The model with $\mathbf{H}=\left(\begin{matrix}
          1&0\\0&1
      \end{matrix}\right)$ and $\mathbf{S}(\zeta)= \left(\begin{matrix}
          0&\zeta\\\zeta&0
      \end{matrix}\right)$, which is Type \Rmnum{5}.A in Table \ref{table of defining equations}.
The subalgebra of $\mathfrak{u}(\mathbf{H})\cap \mathfrak{g}_{0,0}$ consisting of infinitesimal symmetries is
      \[
\left\{\left.
\left(
\begin{array}{cc}
i\theta_1 & 0 \\
0 & i\theta_2
\end{array}
\right)
\,\right|\,
\theta_j\in \mathbb{R}
\right\}.
\] 
      \item The model with $\mathbf{H}=\left(\begin{matrix}
          0&1\\1&0
      \end{matrix}\right)$ and $\mathbf{S}(\zeta)= \left(\begin{matrix}
          0&\zeta\\\zeta&0
      \end{matrix}\right)$, which is Type \Rmnum{5}.B in Table \ref{table of defining equations}. The subalgebra of $\mathfrak{u}(\mathbf{H})\cap \mathfrak{g}_{0,0}$ consisting of infinitesimal symmetries is
      \[
\left\{\left.
\left(
\begin{array}{cc}
\theta_1+i\theta_2 & 0 \\
0 & -\theta_1+i\theta_2
\end{array}
\right)
\,\right|\,
\theta_j\in \mathbb{R}
\right\}.
\]
      \item The model with $\mathbf{H}=\left(\begin{matrix}
          1&0\\0&\epsilon
      \end{matrix}\right)$ and $\mathbf{S}(\zeta)= \left(\begin{matrix}
          \zeta&0\\0&0
      \end{matrix}\right)$ where $\epsilon=\pm1$, which are Types \Rmnum{4}.A and \Rmnum{4}.B in Table \ref{table of defining equations}. The subalgebra of $\mathfrak{u}(\mathbf{H})\cap \mathfrak{g}_{0,0}$ consisting of infinitesimal symmetries is
      \[
\left\{\left.
\left(
\begin{array}{cc}
i\theta_1 & 0 \\
0 & i\theta_2
\end{array}
\right)
\,\right|\,
\theta_j\in \mathbb{R}
\right\}.
\]
    \end{enumerate}
    
\end{proposition}

\begin{proposition}\label{table4prop1}
    A $2$-nondegenerate model in $\mathbb{C}^4$ with 1-dimensional subalgebra of $\mathfrak{u}(\mathbf{H})\cap \mathfrak{g}_{0,0}$ consisting of infinitesimal symmetries is locally biholomorphicaly equivalent to one of the following $2$-nondegenerate models:
    \begin{enumerate}
        \item The model with $\mathbf{H}=\left(\begin{matrix}
          1&0\\0&\epsilon
      \end{matrix}\right)$ and $\mathbf{S}(\zeta)= \left(\begin{matrix}
          1&0\\0&\lambda
      \end{matrix}\right)\zeta$ where $\epsilon=\pm1$ and  $\lambda>1$. The subalgebra of $\mathfrak{u}(\mathbf{H})\cap \mathfrak{g}_{0,0}$ consisting of infinitesimal symmetries is
      \[
\left\{\left.
\left(
\begin{array}{cc}
i\theta & 0 \\
0 & i\theta
\end{array}
\right)
\,\right|\,
\theta\in \mathbb{R}
\right\}.
\]
    \item The model with $\mathbf{H}=\left(\begin{matrix}
          0&1\\1&0
      \end{matrix}\right)$ and $\mathbf{S}(\zeta)= \left(\begin{matrix}
          1&0\\0&e^{i\theta}
      \end{matrix}\right)\zeta$ where $0<\theta<\pi$. The subalgebra of $\mathfrak{u}(\mathbf{H})\cap \mathfrak{g}_{0,0}$ consisting of infinitesimal symmetries is
      \[
\left\{\left.
\left(
\begin{array}{cc}
i\theta & 0 \\
0 & i\theta
\end{array}
\right)
\,\right|\,
\theta\in \mathbb{R}
\right\}.
\]
      \item The model with $\mathbf{H}=\left(\begin{matrix}
          1&0\\0&\epsilon
      \end{matrix}\right)$ and $\mathbf{S}(\zeta)= \left(\begin{matrix}
          0&\zeta\\\zeta&c_{2,\mu}\zeta^\mu
      \end{matrix}\right)$ where $\epsilon=\pm1$, $\mu>1$ is an integer, and $0<c_{2,\mu}$. The subalgebra of $\mathfrak{u}(\mathbf{H})\cap \mathfrak{g}_{0,0}$ consisting of infinitesimal symmetries is
      \[
\left\{\left.
\left(
\begin{array}{cc}
\frac{2-\mu}{\mu}i\theta & 0 \\
0 & i\theta
\end{array}
\right)
\,\right|\,
\theta\in \mathbb{R}
\right\}.
\]
      \item The model with $\mathbf{H}=\left(\begin{matrix}
          1&0\\0&\epsilon
      \end{matrix}\right)$ and $\mathbf{S}(\zeta)= \left(\begin{matrix}
          \zeta&c_{2,\mu}\zeta^\mu\\c_{2,\mu}\zeta^\mu&c_{1,2\mu-1}\zeta^{2\mu-1}
      \end{matrix}\right)$ where $\mu>1$, $c_{2,\mu}>0$ and $c_{1,2\mu-1}\in \mathbb{C}$. The subalgebra of $\mathfrak{u}(\mathbf{H})\cap \mathfrak{g}_{0,0}$ consisting of infinitesimal symmetries is
      \[
\left\{\left.
\left(
\begin{array}{cc}
i\theta & 0 \\
0 & i(2\mu-1)\theta
\end{array}
\right)
\,\right|\,
\theta\in \mathbb{R}
\right\}.
\]
      \item The model with $\mathbf{H}=\left(\begin{matrix}
          1&0\\0&\epsilon
      \end{matrix}\right)$ and $\mathbf{S}(\zeta)= \left(\begin{matrix}
          \zeta&0\\0&c_{1,\mu}\zeta^\mu
      \end{matrix}\right)$ where $\mu>1, c_{1,\mu}>0$. The subalgebra of $\mathfrak{u}(\mathbf{H})\cap \mathfrak{g}_{0,0}$ consisting of infinitesimal symmetries is
      \[
\left\{\left.
\left(
\begin{array}{cc}
i\theta & 0 \\
0 & i\mu\theta
\end{array}
\right)
\,\right|\,
\theta\in \mathbb{R}
\right\}.
\]
      \item The model with $\mathbf{H}=\left(\begin{matrix}
          0&1\\1&0
      \end{matrix}\right)$ and $\mathbf{S}(\zeta)= \left(\begin{matrix}
          \zeta&\zeta\\\zeta&0
      \end{matrix}\right)$. The subalgebra of $\mathfrak{u}(\mathbf{H})\cap \mathfrak{g}_{0,0}$ consisting of infinitesimal symmetries is
      \[
\left\{\left.
\left(
\begin{array}{cc}
i\theta & 0 \\
0 & i\theta
\end{array}
\right)
\,\right|\,
\theta\in \mathbb{R}
\right\}.
\]

      \end{enumerate}
    
\end{proposition}

\begin{proposition}\label{table4prop}
    A $2$-nondegenerate model in $\mathbb{C}^4$ with 0-dimensional subalgebra of $\mathfrak{u}(\mathbf{H})\cap \mathfrak{g}_{0,0}$ consisting of infinitesimal symmetries is locally biholomorphicaly equivalent to a representative in Table \ref{normal forms table1} or \ref{normal forms table2}.
\end{proposition}

\begin{table}[htbp]
\scalebox{0.85}{
    \begin{tabular}{|c|c|c|}\hline
      $\mathbf{H}$ & First nonzero terms in the expansion on $\mathbf{S}(\zeta)$ & \multirow{2}*{General term}\\\cdashline{1-2}
       \To{3.5}\Bo{2.5}\parbox{2.5cm}{\centering Where formula is derived}& Possible normal forms& \\\hline\hline

     \To{5}\Bo{4} 
     $\left(\begin{matrix}
          1&0\\0&\epsilon
      \end{matrix}\right)$
      &$\left(\begin{matrix}
          1&0\\0&\lambda
      \end{matrix}\right)\zeta+\left(\begin{matrix}
          0&\tau(1-\lambda^2)\\\tau(1-\lambda^2)&o_1
      \end{matrix}\right)\zeta^2+\dots$ & 
      \multirow{2}*{
      $\left(\begin{matrix}
          0&c_{2,k}\\ c_{2,k}&c_{1,k}
      \end{matrix}\right)$
      }
      \\\cdashline{1-2} 
      \S \ref{row cases subsection a}& either $\tau> 0$ or  $\tau= 0,o_1>0$&\\
      \hline

     \To{5}\Bo{4} 
     $\left(\begin{matrix}
          1&0\\0&\epsilon
      \end{matrix}\right)$
      &$\left(\begin{matrix}
          1&0\\0&\lambda
      \end{matrix}\right)\zeta+\left(\begin{matrix}
          0&c_{2,\mu}\\ c_{2,\mu}&c_{1,\mu}
      \end{matrix}\right)\zeta^\mu+\dots $ &
      \multirow{2}*{
      $\left(\begin{matrix}
          0&c_{2,k}\\ c_{2,k}&c_{1,k}
      \end{matrix}\right)$
      }
      \\\cdashline{1-2}  
      \S \ref{row cases subsection a}& either $c_{1,\mu}>0$ or $c_{1,\mu}=0,c_{2,\mu}>0$&\\
      \hline

    \To{5}\Bo{4} 
      $\left(\begin{matrix}
          0&1\\1&0
      \end{matrix}\right)$
      &$\left(\begin{matrix}
          1&0\\0&e^{i\theta}
      \end{matrix}\right)\zeta+\left(\begin{matrix}
          0&\tau (1-e^{2i\theta})\\
          \tau (1-e^{2i\theta})&o_1
      \end{matrix}\right)\zeta^2+\dots $ &
      \multirow{2}*{
      $\left(\begin{matrix}
          0&c_{2,k}\\ c_{2,k}&c_{1,k}
      \end{matrix}\right)$
      }
      \\\cdashline{1-2} 
      \S \ref{row cases subsection a}& either $\tau>0$ or $\tau= 0,o_1>0$&\\
      \hline

     \To{5}\Bo{4} 
     $\left(\begin{matrix}
          0&1\\1&0
      \end{matrix}\right)$
      &$\left(\begin{matrix}
          1&0\\0&e^{i\theta}
      \end{matrix}\right)\zeta+\left(\begin{matrix}
          0&c_{2,\mu}\\ c_{2,\mu}&c_{1,\mu}
      \end{matrix}\right)\zeta^\mu+\dots$  &
      \multirow{2}*{
      $\left(\begin{matrix}
          0&c_{2,k}\\ c_{2,k}&c_{1,k}
      \end{matrix}\right)$
      }
      \\\cdashline{1-2}  
      \S \ref{row cases subsection a}& \Bo{1.2}either $c_{1,\mu}>0$ or $c_{1,\mu}=0,c_{2,\mu}>0$&\\\hline

     \To{5}\Bo{4} 
     $\left(\begin{matrix}
          1&0\\0&\epsilon
      \end{matrix}\right)$
      &$\left(\begin{matrix}
          0&1\\1&0
      \end{matrix}\right)\zeta+\left(\begin{matrix}
          o_1&0\\0&o_2
      \end{matrix}\right)\zeta^2+\left(\begin{matrix}
          c_{1,\mu}&0\\0&c_{2,\mu}
      \end{matrix}\right)\zeta^\mu+\dots $ &
      \multirow{2}*{$\left(\begin{matrix}
          c_{1,k}&0\\ 0&c_{2,k}
      \end{matrix}\right)$
      }
      \\\cdashline{1-2} 
      \S \ref{row cases subsection b} &\To{4}\Bo{3} 
      \parbox{10cm}{\centering either $o_2\geq o_1>0$ or\\ $o_1=0,o_2>0$ and either $|c_{2,\mu}|\geq c_{1,\mu}>0$ or $c_{2,\mu}> c_{1,\mu}=0$}&\\ \hline

        \To{5}\Bo{4} 
     $\left(\begin{matrix}
          1&0\\0&\epsilon
      \end{matrix}\right)$
      &$\left(\begin{matrix}
          0&1\\1&0
      \end{matrix}\right)\zeta +\left(\begin{matrix}
      c_{1,\mu}&0\\ 0&c_{2,\mu}     \end{matrix}\right)\zeta^\mu+\left(\begin{matrix}
          c_{1,\mu^\prime}&0\\0&c_{2,\mu^\prime}
      \end{matrix}\right)\zeta^{\mu^\prime}+\dots$  &
      \multirow{2}*{$\left(\begin{matrix}
          c_{1,k}&0\\ 0&c_{2,k}
      \end{matrix}\right)$
      }
      \\\cdashline{1-2}  
      \S \ref{row cases subsection b} &  \To{4}\Bo{3}\parbox{11cm}{\centering either $c_{2,\mu}\geq c_{1,\mu}>0$ or \\ $c_{1,\mu}=0,c_{2,\mu}>0$ and either $|c_{2,\mu^\prime}|\geq c_{1,\mu^\prime}>0$ or $c_{2,\mu^\prime}> c_{1,\mu^\prime}=0$}&\\ \hline

        \To{5}\Bo{4} 
      $\left(\begin{matrix}
          0&1\\1&0
      \end{matrix}\right)$
      &$\left(\begin{matrix}
          0&1\\1&0
      \end{matrix}\right)\zeta+\left(\begin{matrix}
          o_1&0\\0&o_2
      \end{matrix}\right)\zeta^2+\dots $  &
      \multirow{2}*{$\left(\begin{matrix}
          c_{1,k}&0\\ 0&c_{2,k}
      \end{matrix}\right)$
      }
      \\\cdashline{1-2} 
      \S \ref{row cases subsection d}& either $|o_2|\geq o_1=1$ or $o_1=0,o_2=1$& \\ \hline

       \To{5}\Bo{4} 
      $\left(\begin{matrix}
          0&1\\1&0
      \end{matrix}\right)$
      &$\left(\begin{matrix}
          0&1\\1&0
      \end{matrix}\right)\zeta +\left(\begin{matrix}
          c_{1,\mu}&0\\ 0&c_{2,\mu}
      \end{matrix}\right)\zeta^\mu+\dots$  &
      \multirow{2}*{$\left(\begin{matrix}
          c_{1,k}&0\\ 0&c_{2,k}
      \end{matrix}\right)$
      }
      \\\cdashline{1-2} 
      \S \ref{row cases subsection d}& either $|c_{2,\mu}|\geq c_{1,\mu}=1$ or $c_{1,\mu}=0,c_{2,\mu}=1$& \\
    \hline

           \To{5}\Bo{4} 
      $\left(\begin{matrix}
          0&1\\1&0
      \end{matrix}\right)$
      &$\left(\begin{matrix}
          1&1\\1&0
      \end{matrix}\right)\zeta+\left(\begin{matrix}
          0&-4\tau\\ -4\tau&o_1
      \end{matrix}\right)\zeta^2+\dots $  &
      \multirow{2}*{$\left(\begin{matrix}
          0&c_{2,k}\\ c_{2,k}&c_{1,k}
      \end{matrix}\right)$
      }
      \\\cdashline{1-2} 
      \S \ref{row cases subsection a}& either $\tau>0$ or $\tau= 0,o_1>0$& \\
        \hline

      \To{5}\Bo{4} 
      $\left(\begin{matrix}
          0&1\\1&0
      \end{matrix}\right)$
      &$\left(\begin{matrix}
          1&1\\1&0
      \end{matrix}\right)\zeta+\left(\begin{matrix}
          0&c_{2,\mu}\\ c_{2,\mu}&c_{1,\mu}
      \end{matrix}\right)\zeta^\mu+\dots$  &
      \multirow{2}*{$\left(\begin{matrix}
          0&c_{2,k}\\ c_{2,k}&c_{1,k}
      \end{matrix}\right)$
      }
      \\\cdashline{1-2}  
      \S \ref{row cases subsection a}& either $c_{1,\mu}>0$ or $c_{1,\mu}=0,c_{2,\mu}>0$& \\\hline
    \end{tabular}
}
    \captionsetup{width=.9\linewidth}
    \caption{Complete normal form data of $7$-dimensional $2$-nondegenerate hypersurfaces of symbol type as in rows 1, 2, 3, 5, and 6 of Table \ref{7-d. symbol table} whose symmetries in \eqref{CU quotient} form a discrete group. The \emph{general term} column shows the general form of the terms in the expansion by $\zeta^k$, using labels $c_{1,k},c_{2,k}\in \mathbb{C}$. The parameters are $\mu>2$, $\epsilon=\pm1,\lambda>1,0<\theta<\pi$, and $\tau,o_1,o_2,c_{l,k}\in\mathbb{C}$ if not stated otherwise. $\tau$ determines the modified symbol according to Table \ref{7-d. symbol table} and $o_1,o_2$ are obstructions to first order constancy.}
    \label{normal forms table1}
\end{table}

\begin{table}[htbp]
\scalebox{0.9}{
\begin{tabular}{|c|c|}\hline
    $\mathbf{H}$ & First nonzero terms in the expansion on $\mathbf{S}(\zeta)$ that are normalized \\\cdashline{1-2}
       \To{3.5}\Bo{2.5}\parbox{2.5cm}{\centering Where formula is derived}
    & possible normal forms\\
    
    \hline
    \To{5}\Bo{4} 
    $\left(\begin{matrix}
      1&0\\0&\epsilon
    \end{matrix}\right)$
    &$\left(\begin{matrix}
      1&0\\0&0
    \end{matrix}\right)\zeta+\left(\begin{matrix}
      0&\tau\\
      \tau&o_1
    \end{matrix}\right)\zeta^2+\dots$  \\\cdashline{1-2}
    \S \ref{row cases subsection c}& $\tau>0,o_1>0$\\

    \hline
    \To{5}\Bo{4} 
    $\left(\begin{matrix}
      1&0\\0&\epsilon
    \end{matrix}\right)$
    &$\left(\begin{matrix}
      1&0\\0&0
    \end{matrix}\right)\zeta+\left(\begin{matrix}
      0&\tau\\
      \tau&o_1
    \end{matrix}\right)\zeta^2 +\left(\begin{matrix}
      0&c_{2,\mu^\prime}\\ c_{2,\mu^\prime}&c_{1,\mu^\prime}
    \end{matrix}\right)\zeta^\mu+\dots$  \\\cdashline{1-2}
    \S \ref{row cases subsection c}& \To{4}\Bo{2.8}\parbox{10cm}{\centering either $\tau=0,o_1>0$ or $\tau>0,o_1=0$ and either $c_{1,\mu^\prime}>0$ or $c_{2,\mu^\prime}>0$ if $c_{1,\mu^\prime}=0$ or $\mu^\prime=3$ in $\tau=0$ case}\\

    \hline
    \To{5}\Bo{4} 
    $\left(\begin{matrix}
      1&0\\0&\epsilon
    \end{matrix}\right)$
    &$\left(\begin{matrix}
      1&0\\0&0
    \end{matrix}\right)\zeta+\left(\begin{matrix}
      0&\tau\\
      \tau&0
    \end{matrix}\right)\zeta^2+\left(\begin{matrix}
      0&0\\
      0&c_{1,3}
    \end{matrix}\right)\zeta^3+\left(\begin{matrix}
      0&c_{2,\mu^{\prime\prime}}\\ c_{2,\mu^{\prime\prime}}&c_{1,\mu^{\prime\prime}}\end{matrix}\right)\zeta^{\mu^{\prime\prime}}+\dots$  
    \\\cdashline{1-2}
    \S \ref{row cases subsection c} &\Bo{1.2} $\tau>0$ and either $c_{1,\mu^{\prime\prime}}>0$ or $c_{1,\mu^{\prime\prime}}=0,c_{2,\mu^{\prime\prime}}>0$\\

    \hline
    \To{5}\Bo{4} 
    $\left(\begin{matrix}
      1&0\\0&\epsilon
    \end{matrix}\right)$
    &$\left(\begin{matrix}
      1&0\\0&0
    \end{matrix}\right)\zeta+
    \left(\begin{matrix}
      0&c_{2,\mu}\\ c_{2,\mu}&c_{1,\mu}
    \end{matrix}\right)\zeta^\mu+\dots$  \\\cdashline{1-2}
    \S \ref{row cases subsection c}& \Bo{1.2}$c_{1,\mu}>0,c_{2,\mu}>0$ \\

    \hline
    \To{5}\Bo{4} 
    $\left(\begin{matrix}
      1&0\\0&\epsilon
    \end{matrix}\right)$
    &$\left(\begin{matrix}
      1&0\\0&0
    \end{matrix}\right)\zeta +\left(\begin{matrix}
      0&c_{2,\mu}\\ c_{2,\mu}&c_{1,\mu}
    \end{matrix}\right)\zeta^\mu +\left(\begin{matrix}
      0&c_{2,\mu^\prime}\\ c_{2,\mu^\prime}&c_{1,\mu^\prime}
    \end{matrix}\right)\zeta^{\mu^\prime}+\dots$  \\\cdashline{1-2}
    \S \ref{row cases subsection c}& \To{3.6}\Bo{2.6}\parbox{11cm}{\centering either $c_{1,\mu}=0,c_{2,\mu}>0$ or $c_{1,\mu}>0,c_{2,\mu}=0$ and either $c_{1,\mu^\prime}>0$ or $c_{2,\mu^\prime}>0$ if $c_{1,\mu^\prime}=0$ or $\mu^\prime=2\mu-1$ in $c_{1,\mu}=0$ case}\\

    \hline
    \To{5}\Bo{4} 
    $\left(\begin{matrix}
      1&0\\0&\epsilon
    \end{matrix}\right)$
    &$\left(\begin{matrix}
      1&0\\0&0
    \end{matrix}\right)\zeta +\left(\begin{matrix}
      0&c_{2,\mu}\\ c_{2,\mu}&0
    \end{matrix}\right)\zeta^\mu +\left(\begin{matrix}
      0&0\\ 0&c_{1,2\mu-1}   \end{matrix}\right)\zeta^{2\mu-1}+\left(\begin{matrix}
      0&c_{2,\mu^{\prime\prime}}\\ c_{2,\mu^{\prime\prime}}&c_{1,\mu^{\prime\prime}}    \end{matrix}\right)\zeta^{\mu^{\prime\prime}}+\dots$  \\\cdashline{1-2}
    \S \ref{row cases subsection c}&\Bo{1.2} $c_{2,\mu}>0$ and either $c_{1,\mu^{\prime\prime}}>0$ or $c_{1,\mu^{\prime\prime}}=0,c_{2,\mu^{\prime\prime}}>0$
    \\

    \hline
    \To{5}\Bo{4} 
    $\left(\begin{matrix}
      0&1\\1&0
    \end{matrix}\right)$
    &$\left(\begin{matrix}
      1&0\\0&0
    \end{matrix}\right)\zeta+\left(\begin{matrix}
      0&o_2\\o_2&1
    \end{matrix}\right)\zeta^2+\dots$  \\\cdashline{1-2}
    \S \ref{row cases subsection e}& $o_2> 0$ \\

    \hline
    \To{5}\Bo{4} 
    $\left(\begin{matrix}
      0&1\\1&0
    \end{matrix}\right)$
    &$\left(\begin{matrix}
      1&0\\0&0
    \end{matrix}\right)\zeta+\left(\begin{matrix}
      0&1\\1&0
    \end{matrix}\right)\zeta^2+\left(\begin{matrix}
      0&c_{2,3}\\c_{2,3}&0    \end{matrix}\right)\zeta^3+\dots+\left(\begin{matrix}
      0&c_{2,\mu^\prime}\\ c_{2,\mu^\prime}&c_{1,\mu^\prime}
    \end{matrix}\right)\zeta^{\mu^\prime}+\dots$  \\\cdashline{1-2}
    \S \ref{row cases subsection e}&\To{3}\Bo{2} either $c_{1,\mu^\prime}\neq 0, \frac{c_{1,\mu^\prime+1}}{c_{1,\mu^\prime}}\in \mathbb{R}$ or $c_{1,\mu^{\prime}}=0$ for all $\mu^{\prime},c_{2,3}\in \mathbb{R}$\\

    \hline
    \To{5}\Bo{4} 
    $\left(\begin{matrix}
      0&1\\1&0
    \end{matrix}\right)$
    &$\left(\begin{matrix}
      1&0\\0&0
    \end{matrix}\right)\zeta+\left(\begin{matrix}
      0&c_{2,\mu}\\ c_{2,\mu}&1
    \end{matrix}\right)\zeta^\mu+\dots$  \\\cdashline{1-2}
    \S \ref{row cases subsection e}& \Bo{1.2}$c_{2,\mu}> 0$ \\

    \hline
    \To{5}\Bo{4} 
    $\left(\begin{matrix}
      0&1\\1&0
    \end{matrix}\right)$
    &$\left(\begin{matrix}
      1&0\\0&0
    \end{matrix}\right)\zeta+\left(\begin{matrix}
      0&1\\1&0    \end{matrix}\right)\zeta^\mu+\left(\begin{matrix}
      0&c_{2,\mu+1}\\ c_{2,\mu+1}&0
    \end{matrix}\right)\zeta^\mu+\dots +\left(\begin{matrix}
      0&c_{2,\mu^\prime}\\ c_{2,\mu^\prime}&c_{1,\mu^\prime}
    \end{matrix}\right)\zeta^{\mu^\prime}+\dots$  \\\cdashline{1-2}
    \S \ref{row cases subsection e}& \To{3}\Bo{2} either $c_{1,\mu^\prime}\neq 0, \frac{c_{1,\mu^\prime+\mu-1}}{c_{1,\mu^\prime}}\in \mathbb{R}$ or $c_{1,\mu^{\prime}}=0$ for all $\mu^{\prime},c_{2,2\mu-1}\in \mathbb{R}$\\
    \hline
\end{tabular}
}
    \captionsetup{width=.9\linewidth}
    \caption{Complete normal form data of $7$-dimensional $2$-nondegenerate hypersurfaces of symbol type as in rows 4 and 7 of Table \ref{7-d. symbol table} whose symmetries in \eqref{CU quotient} form a discrete group. Terms in the expansion are all of the form $\left(\begin{matrix}
          0&c_{2,k}\\ c_{2,k}&c_{1,k}
      \end{matrix}\right)\zeta^k$ for $c_{1,k},c_{2,k}\in \mathbb{C}$.   The parameters are $\mu^\prime,\mu^{\prime\prime}>\mu>2$, $\epsilon=\pm 1$, and $\tau,o_1,o_2,c_{l,k}\in\mathbb{C}$ if not stated otherwise. $\tau$ determines the modified symbol according to Table \ref{7-d. symbol table} and $o_1,o_2$ are obstructions to first order constancy.}
    \label{normal forms table2}
\end{table}

The proof of these propositions starts by transforming \eqref{gen def fun} so that its bigraded symbol representatives $(\mathbf{H},S^{0,2})$ in the new coordinates are in the normal form of Table \ref{normalization conditions Table}. Next, we will apply transformations of the form \eqref{symgrouptransform} with $(u,U)$ in \eqref{CU quotient}. 
Since $\mathfrak{g}^{\prime}_{0,0}$ is the complexification of the Lie algebra of \eqref{CU quotient}, exponentiating the $\mathfrak{g}^{\prime}_{0,0}$ matrices of Table \ref{7-d. symbol table} provides generators of the identity component in \eqref{CU quotient}. We find this group's other components by direct computation, and we consider a few separate cases corresponding to the separate rows of  Table \ref{7-d. symbol table}. This will be done in the separate subsections.

We adopt the following uniform conventions for labeling elements in the expansion of $\mathbf{S}(\zeta)$ with respect to $\zeta$.
Since we bring $\mathbf{S}_{\zeta\zeta}(0)$  to normal form 
\[
O^N_{0,2}+\Omega\mathbf{S}_\zeta(0)+\mathbf{S}_\zeta(0)\Omega^T
\]
according to Proposition \ref{degree 0 symmetry bound} 
for $\Omega$ (depending on $\tau$) from Table \ref{7-d. symbol table} (except for the case in the sixth row of Table \ref{7-d. symbol table}, where we instead bring it to $\mathbf{S}_{\zeta\zeta}(0)=O^N_{0,2}+\Omega\mathbf{S}_\zeta(0)+\mathbf{S}_\zeta(0)\Omega^T-6\tau\mathbf{S}_\zeta(0)$ to obtain the claimed normal form), we label $o_1,o_2$ the variables appearing in the obstruction $O^N_{0,2}$ to first order constancy in the image of $\pi_{N_{0,2}}$ from Table \ref{normalization conditions Table}. In the remainder of the expansion, we have in degree $\mu$ two terms with complex coefficients $c_{1,\mu}$ and $c_{2,\mu}$, but their explicit position in the symmetric matrix depends on the bigraded symbol and we present it in the Tables \ref{normal forms table1} and \ref{normal forms table2}.

\subsection{Expansions of Row 1, 2, and 6 in Table 3}\label{row cases subsection a}

For the rows $1$, $2$, or $6$, if $\tau>0$, then only the trivial subgroup in the identity component of \eqref{CU quotient} preserves $\tau>0$. For row $6$, there are no further connected components and we are done with this case. For the rows $1$ and $2$, there is another connected component of \eqref{CU quotient} represented by $(1,\left(\begin{smallmatrix}
    1& 0\\ 0& -1
\end{smallmatrix}\right))$. There is an element in this other component of \eqref{CU quotient} mapping $\tau\mapsto \tau, o_1\mapsto -o_1, c_{1,\mu}\mapsto -e^{i\pi(\mu-1)}c_{1,\mu}$ and $c_{2,\mu}\mapsto e^{i\pi(\mu-1)}c_{2,\mu}$, which can be used to normalize the next nonzero element in the expansion of $\mathbf{S}(\zeta)$ according to Remark \ref{discrete normalization remark}, finishing the normalization.

If $\tau=0$, then the identity component of \eqref{CU quotient} consists of unit multiples of the identity matrix, so our transformations are given by  $u=1$, $U=e^{i\theta}\mathrm{Id}$ and $g_U(\zeta)=(e^{2i\theta}-1)\zeta$ for $\theta\in \mathbb{R}$. Such transformations have the effect of rescaling the terms of degree $\mu$ in the expansion of $\mathbf{S}(\zeta)$ by $e^{i(2\mu-2)\theta}$. Considering the first nonzero term in the expansion of $\mathbf{S}(\zeta)$ with degree greater than $1$ we can transform it to a positive real number.  If there is no such term then our transformations are all symmetries. Otherwise, the subgroup of the identity component in \eqref{CU quotient} preserving this last normalization is the discrete group 
\begin{align}\label{r126 discrete group}
\{(u,U)=(1,e^{i\theta}\mathrm{Id})\,|\, \theta(2\mu-2)=2k\pi,\,\forall\,k=0,\ldots,\mu-2\}.
\end{align}
\begin{remark}\label{exhausting finite group}
    In the following paragraph, we describe a normalization procedure that exhausts the remaining freedom from the action of the finite group \eqref{r126 discrete group}. We apply an analogous procedure in latter subsections as well, but do not repeat the details there.
\end{remark}
Applying this discrete group, the next nonzero coefficient $c_{j,\mu^\prime}$ in the expansion of $\mathbf{S}(\zeta)$ with  $\mu^\prime>\mu$ can be transformed to any number of the form $e^{i(2\mu^\prime-2)2k/(2\mu-2)}c_{j,\mu^\prime}$, normalizing it according to Remark \ref{discrete normalization remark}. Still a nontrivial subgroup of \eqref{r126 discrete group} can preserve this additional normalization, in which case we apply this subgroup's action to normalize the next nontrivial term in the expansion of $\mathbf{S}(\zeta)$ of degree $\mu^{\prime\prime}>\mu^\prime$, essentially repeating the previous step in a higher degree. Next, we repeat this process, iteratively using part of \eqref{r126 discrete group} to normalize terms of subsequently higher degree until all of \eqref{r126 discrete group} is exhausted (or there remains only a subgroup of \eqref{r126 discrete group} acting by symmetries on $\mathbf{S}(\zeta)$, a special case that is easily recognized).

 For row $6$, there are no further connected components of \eqref{CU quotient} and we are done with this case. For the rows $1$ and $2$, there is another connected component of \eqref{CU quotient} represented by $(1,\left(\begin{smallmatrix}
    1& 0\\ 0& -1
\end{smallmatrix}\right))$. Therefore, we have two possibilities how this other component of \eqref{CU quotient} can act, while preserving the normalizations we achieved so far. If $c_{1,\mu}>0$, then we have an element in the other connected component of \eqref{CU quotient} mapping $c_{1,\mu}\mapsto c_{1,\mu}, c_{2,\mu}\mapsto -c_{2,\mu}, c_{1,\mu^\prime}\mapsto -e^{i\pi\frac{\mu^\prime-1}{\mu-1}}c_{1,\mu^\prime}$, and $c_{2,\mu^\prime}\mapsto e^{i\pi\frac{\mu^\prime-1}{\mu-1}}c_{2,\mu^\prime}$ which we may use to finish the normalization according to Remark \ref{discrete normalization remark}. If $c_{1,\mu}=0,c_{2,\mu}>0$, then we have an element in this other component of \eqref{CU quotient} mapping $c_{2,\mu}\mapsto c_{2,\mu}, c_{1,\mu^\prime}\mapsto -c_{1,\mu^\prime}$ and $c_{2,\mu^\prime}\mapsto c_{2,\mu^\prime}$ with which to finish the normalization. The remaining transformations (if any) preserving this additional normalization are symmetries, so we are done with these cases from rows 1, 2, or 6 of Table \ref{7-d. symbol table}.

\subsection{Expansions of Row 3 in Table 3}\label{row cases subsection b}

Suppose that $(\mathbf{H},S^{0,2})$ is as in row $3$ of Table \ref{7-d. symbol table}, with the parameter $\epsilon=\pm1$. The connected component of the identity in \eqref{CU quotient} consists of matrices of the form 
\begin{align}\label{s6 T3 iso group}
\left(
\begin{array}{cc}
     e^{i\theta_1}& 0 \\
     0&e^{i\theta_2}
\end{array}
\right)
\end{align}
with $\theta_1,\theta_2$ real. There is another component represented by  $(i^{(1-\epsilon)/2},\left(\begin{smallmatrix}
    0& 1\\ 1& 0
\end{smallmatrix}\right))$.

Let $\mu>1$ and $\mu^\prime>\mu$ be integers satisfying the following three properties: First, the expansion of $\mathbf{S}(\zeta)$ has the form
\begin{align}\label{first super linear terms R3 case a}
\mathbf{S}(\zeta)=\zeta\mathbf{S}_\zeta(0)+
\left(
\begin{array}{cc}
     \gamma_1 &0 \\
    0 & \gamma_2
\end{array}
\right)\zeta^\mu+
\left(
\begin{array}{cc}
     \gamma_1^\prime &0 \\
    0 & \gamma_2^\prime
\end{array}
\right)\zeta^{\mu^\prime}+O(\zeta^{\mu^\prime+1}),
\end{align} 
where we can use an element in the non-identity connected component of \eqref{CU quotient} to achieve $|\gamma_1|\leq |\gamma_2|$ or simultaneously $|\gamma_1|= |\gamma_2|$ and $|\gamma_1^\prime|\leq |\gamma_2^\prime|$.
Second, the coefficients $\gamma_1$ and $\gamma_2$ are both zero if and only if $\mathbf{S}(\zeta)=\zeta\mathbf{S}_\zeta(0)$ (if this is the case, then all elements of \eqref{CU quotient} are symmetries).
And third, the coefficients $\gamma_1^\prime$ and $\gamma_2^\prime$ are both zero if and only if $(\frac{\partial}{\partial \zeta})^{\mu+1}\mathbf{S}=0$.
 
The identity component of \eqref{CU quotient} acts diagonally on $\gamma_j$ mapping $\gamma_1\mapsto e^{i((\mu-2)\theta_1+\mu\theta_2)}\gamma_1$ and $\gamma_2\mapsto e^{i(\mu\theta_1+(\mu-2)\theta_2)}\gamma_2$ and analogously on $\gamma_j^\prime$ with $\mu^\prime$ replacing $\mu$.
 
If $|\gamma_1|>0$, then we can make both $\gamma_j$ real and positive. Only a finite subgroup of the identity component in \eqref{CU quotient} still preserves this additional normalization, namely
\begin{align}\label{r3 discrete group}
\left\{\left.
    \left(1,\left(
    \begin{array}{cc}
     e^{i\theta_1}& 0 \\
     0&e^{i\theta_2}
    \end{array}
    \right)\right)
    \,\right|\,
    \parbox{6cm}{$((\mu-2)\theta_1+\mu\theta_2)=2k_1\pi$,\\ $(\mu\theta_1+(\mu-2)\theta_2)=2k_2\pi,\,\forall\,k_j\in\mathbb{Z}$}
    \right\}.
\end{align}
Following the procedure referred to in Remark \ref{exhausting finite group}, one can exhaust all of \eqref{r3 discrete group}, further normalizing $\mathbf{S}(\zeta)$.

In the special case when $|\gamma_1|= |\gamma_2|$, we still have freedom to act by elements from the non-identity connected component of \eqref{CU quotient} by finding the first occurrence of $|c_{1,\mu^{\prime\prime}}|\neq |c_{2,\mu^{\prime\prime}}|$ and exhaust this freedom by ordering them so that $|c_{1,\mu^{\prime\prime}}|< |c_{2,\mu^{\prime\prime}}|$. If there is no such occurrence, then we have a discrete symmetry provided by the above representative of the non-identity connected component of \eqref{CU quotient}.
 
If $\gamma_1=0$, then we  can make $\gamma_2>0$ and restrict to the subgroup where $\theta_1=\frac{2k_2\pi-(\mu-2)\theta_2}{\mu}$ with $k_2\in\mathbb{Z}$. Then the action on $\gamma_j^\prime$ becomes $\gamma_1^\prime\mapsto e^{2i\frac{(\mu^\prime+\mu-2)\theta_2+k_2\pi(\mu^\prime-2)}{\mu}}\gamma_1^\prime$ and $\gamma_2^\prime\mapsto e^{2i\frac{(\mu^\prime-\mu)\theta_2+k_2\pi\mu^\prime}{\mu}}\gamma_2^\prime.$ Thus if there are no $\gamma_j^\prime$, then we have $1$-dimensional freedom of symmetries, and otherwise we can exhaust the freedom by putting $\gamma_1^\prime>0$ or $\gamma_1^\prime=0,\gamma_2^\prime>0.$ What remains is a finite group that one can exhaust following the procedure referred to in Remark \ref{exhausting finite group}.

\subsection{Expansions of Row 4 in Table 3}\label{row cases subsection c}
  
Next suppose that $(\mathbf{H},S^{0,2})$ is as in row $4$ of Table \ref{7-d. symbol table}. In this case the entire group \eqref{CU quotient} is comprised of matrices of the form
\[
\left(
\begin{array}{cc}
     e^{i\theta_1} & 0 \\
     0 & e^{i\theta_2}
\end{array}
\right)
\quad\quad\forall\,\theta_j\in\mathbb{R}.
\]

Let $\mu>1$ and $\mu^\prime>\mu$ be integers satisfying the following three properties: First, the expansion of $\mathbf{S}(\zeta)$ has the form
\begin{align}\label{first super linear terms R3 case a}
\mathbf{S}(\zeta)=\zeta\mathbf{S}_\zeta(0)+
\left(
\begin{array}{cc}
     0 &\gamma_2 \\
    \gamma_2 & \gamma_1
\end{array}
\right)\zeta^\mu+
\left(
\begin{array}{cc}
     0 &\gamma_2^\prime \\
    \gamma_2^\prime & \gamma_1^\prime
\end{array}
\right)\zeta^{\mu^\prime}+O(\zeta^{\mu^\prime+1}).
\end{align} 
Second, the coefficients $\gamma_1$ and $\gamma_2$ are both zero if and only if $\mathbf{S}(\zeta)=\zeta\mathbf{S}_\zeta(0)$ (if this is the case, then all elements of \eqref{CU quotient} are symmetries).
And third, the coefficients $\gamma_1^\prime$ and $\gamma_2^\prime$ are both zero if and only if $(\frac{\partial}{\partial \zeta})^{\mu+1}\mathbf{S}=0$.

The group action of \eqref{CU quotient} on $\mathbf{S}(\zeta)$ preserves the general form of \eqref{first super linear terms R7 case a} while mapping $\gamma_1\mapsto e^{2i(\mu\theta_1-\theta_2)}\gamma_1$ and $\gamma_2\mapsto e^{i((2\mu-1)\theta_1-\theta_2)}\gamma_2$ and acting analogously on $\gamma_j^\prime$ with $\mu$ replaced by $\mu^\prime$. Thus if  $\gamma_1\neq 0$ and $\gamma_2\neq 0$, then we can make $\gamma_1>0$ and $\gamma_2>0.$ Only a finite subgroup of the identity component in \eqref{CU quotient} still preserves this additional normalization, namely
\begin{align}\label{r4 discrete group}
\left\{\left.
    \left(1,\left(
    \begin{array}{cc}
     e^{i\theta_1}& 0 \\
     0&e^{i\theta_2}
    \end{array}
    \right)\right)
    \,\right|\,
    \parbox{5.6cm}{$\mu\theta_1-\theta_2=k_1\pi$,\\ $(2\mu-1)\theta_1-\theta_2=2k_2\pi,\,\forall\,k_j\in\mathbb{Z}$}
    \right\}.
\end{align}
Following the procedure referred to in Remark \ref{exhausting finite group}, one can exhaust all of \eqref{r4 discrete group}, finishing the normalization of $\mathbf{S}(\zeta)$ for the $\gamma_1\neq 0$ and $\gamma_2\neq 0$ case.

If $\gamma_1=0$, then we can make $\gamma_2>0$ and there remains freedom of action by the subgroup with $\theta_2=(2\mu-1)\theta_1.$ This acts as $\gamma_1^\prime\mapsto e^{2i(\mu^\prime-2\mu+1)\theta_1}\gamma_1^\prime$ and $\gamma_2^\prime\mapsto e^{2i(\mu^\prime-\mu)\theta_1}\gamma_2^\prime$. Thus except the case $\mu^\prime=2\mu-1$ and $\gamma_2^\prime=0$, we can make the normalization $\gamma_1^\prime>0$ or simultaneously $\gamma_1^\prime=0$ and $\gamma_2^\prime>0$ at this level. In the exceptional case, we can move on to normalize the next nonzero term analogously, and if there is no such term then the remaining group acts by symmetries.

If $\gamma_2=0$, then $\gamma_1>0$ and there remains freedom of action by the subgroup with $\theta_2=\mu\theta_1+k_1\pi, k_1\in \mathbb{Z}.$ This acts as $\gamma_1^\prime\mapsto (-1)^{k_1} e^{2i(\mu^\prime-\mu)\theta_1}\gamma_1^\prime$ and $\gamma_2^\prime\mapsto (-1)^{k_1} e^{i(2\mu^\prime-\mu-1)\theta_1}\gamma_2^\prime$. Thus we can make the normalization of $\gamma_1^\prime>0$ or simultaneously $\gamma_1^\prime=0$ and $\gamma_2^\prime>0$. If there are no $\gamma_j^\prime$, then the group acts by symmetries.

\subsection{Expansions of Row 5 in Table 3}\label{row cases subsection d}

Next suppose that $(\mathbf{H},S^{0,2})$ is as in row $5$ of Table \ref{7-d. symbol table}. Accordingly the identity component of \eqref{CU quotient} is comprised of matrices of the form 
\begin{align}\label{r5 T3 iso group}
\left(
\begin{array}{cc}
     e^{\theta_1+i\theta_2}& 0 \\
     0&e^{-\theta_1+i\theta_2}
\end{array}
\right)
\end{align}
with $\theta_1,\theta_2$ are real. There are another 3 connected components of \eqref{CU quotient} generated by elements $(i,\left(\begin{smallmatrix}
    1& 0\\ 0& -1
\end{smallmatrix}\right))$ and $(1,\left(\begin{smallmatrix}
    0& 1\\ 1&0
\end{smallmatrix}\right)).$

Suppose the expansion of $\mathbf{S}(\zeta)$ has the form
\begin{align}\label{first super linear terms R5}
\mathbf{S}(\zeta)=\zeta\mathbf{S}_\zeta(0)+
\left(
\begin{array}{cc}
     \gamma_1 &0 \\
    0 & \gamma_2
\end{array}
\right)\zeta^\mu+O(\zeta^{\mu+1}),
\end{align} 
where $\gamma_1$ and $\gamma_2$ are both zero if and only if $\mathbf{S}(\zeta)=\zeta\mathbf{S}_\zeta(0)$ (if this is the case, then all elements of \eqref{CU quotient} are symmetries). We can order $|\gamma_1|\leq |\gamma_2|$ by applying the second of the aforementioned generators. 

The group action of \eqref{r5 T3 iso group} on $\mathbf{S}(\zeta)$ preserves the general form of \eqref{first super linear terms R5} while mapping $\gamma_1\mapsto e^{-2\theta_1+2i(\mu-1)\theta_2}\gamma_1$ and $\gamma_2\mapsto e^{2\theta_1+2i(\mu-1)\theta_2}\gamma_2$. Thus we can make $\gamma_1=1$ or simultaneously $\gamma_1=0$ and $\gamma_2=1.$ The situation with the remaining discrete subgroup of the identity component of \eqref{CU quotient} is the same as in Section \ref{row cases subsection a} (i.e., equivalent to \eqref{r126 discrete group} with $\theta_2$ in place of $\theta$), and one can similarly exhaust this freedom following the procedure referred to in Remark \ref{exhausting finite group}.

In the connected component of \eqref{CU quotient} corresponding to the first of the generators, we have an element mapping $\gamma_j\mapsto \gamma_j$, and $c_{j,\mu^\prime}\mapsto -e^{\pi i\frac{\mu^\prime-1}{\mu-1}}c_{j,\mu^\prime}$, which we use to normalize the next nonzero term with order $\mu^\prime>\mu$ according to Remark \ref{discrete normalization remark}. This exhausts the action of this connected component, leaving one final case to consider.

In the special case when $\gamma_1=\gamma_2=1$, we still have freedom to act by elements from the connected component of \eqref{CU quotient} of the second of the generators. Taking the first occurrence of $|c_{1,\mu^{\prime\prime}}|\neq |c_{2,\mu^{\prime\prime}}|$, we exhaust this freedom by ordering them such that $|c_{1,\mu^{\prime\prime}}|< |c_{2,\mu^{\prime\prime}}|$. If there is no such occurrence, then we have a discrete symmetry provided by the second representative above of the connected components in \eqref{CU quotient}.

\subsection{Expansions of Row 7 in Table 3}\label{row cases subsection e}

 Suppose that $(\mathbf{H},S^{0,2})$ is as in row $7$ of Table \ref{7-d. symbol table}. In this case, the identity component of \eqref{CU quotient} is comprised of matrices of the form
\begin{align}\label{r7 tr group}
\left(
\begin{array}{cc}
     e^{i\theta_1-\theta_2} & i\theta_3e^{i\theta_1+\theta_2} \\
     0 & e^{i\theta_1+\theta_2}
\end{array}
\right)
\quad\quad\forall\,\theta_j\in\mathbb{R}.
\end{align}
There is another connected components of \eqref{CU quotient} generated by element $(i,\left(\begin{smallmatrix}
    1& 0\\ 0&-1
\end{smallmatrix}\right)).$

Suppose the expansion of $\mathbf{S}(\zeta)$ has the form
\begin{align}\label{first super linear terms R7 case a}
\mathbf{S}(\zeta)=\zeta\mathbf{S}_\zeta(0)+
\left(
\begin{array}{cc}
     0 & \gamma_2 \\
    \gamma_2 & \gamma_1
\end{array}
\right)\zeta^\mu+O(\zeta^{\mu+1}),
\end{align} 
where $\gamma_1$ and $\gamma_2$ are both zero if and only if $\mathbf{S}(\zeta)=\zeta\mathbf{S}_\zeta(0)$ (if this is the case, then all elements of \eqref{CU quotient} are symmetries).

The group action of \eqref{r7 tr group} on $\mathbf{S}(\zeta)$ preserves the general form of \eqref{first super linear terms R7 case a} while transforming the $\gamma_j$ coefficients. It acts linearly on $\gamma_1,\gamma_2$ by the matrix
\[\left(
\begin{array}{cc}
     e^{2i\theta_1(\mu-1)-(2\mu+2)\theta_2} & 0 \\
    -i\theta_3e^{2i\theta_1(\mu-1)-(2\mu+2)\theta_2} & e^{2i\theta_1(\mu-1)-2\mu\theta_2}
\end{array}
\right).
\]
If $\gamma_1\neq 0$ then we can make $\gamma_1=1$ and $\gamma_2\in \mathbb{R}$ exhausting all the freedom in the action up to a finite subgroup. This subgroup is the same as in Section \ref{row cases subsection a} (i.e., equivalent to \eqref{r126 discrete group} with $\theta_2$ in place of $\theta$), and one can similarly exhaust this freedom following the procedure referred to in Remark \ref{exhausting finite group}.

In the another connected component of \eqref{CU quotient}, we have an element that can make $\gamma_2\geq 0$. In the special case $\gamma_2=0$, we have an element in the non-identity connected component of \eqref{CU quotient} mapping $\gamma_1\mapsto \gamma_1$, $c_{1,\mu^\prime}\mapsto -e^{\pi i\frac{\mu^\prime-1}{\mu-1}}c_{1,\mu^\prime}$, and $c_{2,\mu^\prime}\mapsto e^{\pi i\frac{\mu^\prime-1}{\mu-1}}c_{2,\mu^\prime}$, which we can use to finish the normalization according to Remark \ref{discrete normalization remark}

Otherwise $\gamma_1=0$, in which case we can make $\gamma_2=1$ and restrict to acting by the subgroup with $\theta_1=0$ ande $\theta_2=0.$ 
Proceeding in this latter case, it is advantageous to work in full generality and assume
\begin{align}\label{R7finalCase}
\mathbf{S}(\zeta)=\zeta\mathbf{S}_\zeta(0)+
\left(
\begin{array}{cc}
     0 & f_2(\zeta)\\
    f_2(\zeta) & f_1(\zeta)
\end{array}
\right),
\end{align} 
with $f_2(\zeta)=\zeta^\mu+O(\zeta^{\mu+1})$ and $f_1(\zeta)=\gamma_1^\prime \zeta^{\mu^\prime}+O(\zeta^{\mu^\prime+1})$ for some $\mu<\mu^\prime.$ The action of the subgroup of \eqref{r7 tr group} given by $\theta_1=\theta_2=0$ transforms \eqref{R7finalCase} to 
\begin{align}
\mathbf{S}(\zeta)=\zeta\mathbf{S}_\zeta(0)+
\left(
\begin{array}{cc}
     0 &  -f_1(g(\zeta,\theta_3))i\theta_3+f_2(g(\zeta,\theta_3))\\
    -f_1(g(\zeta,\theta_3))i\theta_3+f_2(g(\zeta,\theta_3)) & f_1(g(\zeta,\theta_3))
\end{array}
\right),
\end{align} 
where $g(\zeta,\theta_3)=\zeta+f_1(g(\zeta,\theta_3))\theta_3^2+2if_2(g(\zeta,\theta_3))\theta_3$ is the uniquely determined function referred to as $g_{1,U}$ in formula \eqref{symgrouptransform}, and we have simply rewritten the dependence on $(1,U)$ as dependence on $\theta_3$. We can see this provides a recurrence formula for real analytic $g(\zeta,\theta_3)$ and in particular, the absolute term in $\theta_3$ is just $\zeta$ and the term linear in $\theta_3$ is $2if_2(\zeta)\theta_3$. 

Thus if $\gamma_1^\prime\neq0$, then the term of $f_1(g(\zeta,\theta_3))$ linear in $\theta_3$ and of lowest order in $\zeta$ is $(\gamma_1^{\prime\prime}+2i\mu^\prime\gamma_1^\prime\theta_3)\zeta^{\mu^\prime+\mu-1}$, where $\gamma_1^{\prime\prime}\zeta^{\mu^\prime+\mu-1}$ is the term from the expansion of $f_1(\zeta).$ Therefore, we can normalize $\gamma_1^{\prime\prime}/\gamma_1^{\prime}\in \mathbb{R}$. In the non-identity connected component of \eqref{CU quotient}, we have an element mapping $\gamma_2\mapsto \gamma_2,\gamma_1^{\prime}\mapsto -\gamma_1^{\prime} $ and $\gamma_1^{\prime\prime}\mapsto -\gamma_1^{\prime\prime}$, which we can use to finish the normalization according to Remark \ref{discrete normalization remark}

If $\gamma_1^\prime=0$, then $f_1(\zeta)=0$ and the term of $f_2(g(\zeta,\theta_3))$ linear in $\theta_3$ and of lowest order in $\zeta$ is $(\gamma_2^{\prime\prime}+2i\mu\theta_3)\zeta^{2\mu-1}$, where $\gamma_2^{\prime\prime}\zeta^{2\mu-1}$ is a term from the expansion of $f_2(\zeta).$ Therefore, we can normalize $\gamma_2^{\prime\prime}\in \mathbb{R}$. The remaining freedom to normalize is in the action of a finite group, namely the group generated by an idempotent element in the non-identity connected component of \eqref{CU quotient} that maps $\gamma_2\mapsto \gamma_2$, $\gamma_2^{\prime\prime} \mapsto \gamma_2^{\prime\prime}$, $c_{1,\mu^{\prime\prime\prime}}\mapsto -c_{1,\mu^{\prime\prime\prime}}$ and $c_{2,\mu^{\prime\prime\prime}}\mapsto c_{2,\mu^{\prime\prime\prime}}$, which we can use to finish the normalization.

\section{Distinguished realizations of modified symbols}\label{Approaches to Hypersurface Realization}\label{7-dimensional realizations}

The normal forms from the previous section provide simple realizations of the modified symbols. We compare these simple realizations with the distinguished $2$-nondegenerate models given by \eqref{new S2} realizing the modified symbols using the data from Table \ref{7-d. symbol table}, in detail. The latter class is especially interesting because finding the structures within it that have constant modified CR symbols turns out to give an alternate method for deriving the homogeneous structure classification in \cite{SykesHomogeneous} and produces the defining equations from Table \ref{table of defining equations}.

The $2$-nondegenerate models given by \eqref{new S2} depend on three distinguished holomorphic functions $f_1,f_2,f_3$ of one variable and the corresponding defining equation \eqref{gen def fun} becomes $\Re(w)=\frac{Q_1}{Q_2}$ for the following three cases depending on $\mathbf{H}$ from Table \ref{7-d. symbol table}.

In the case $
\mathbf{H}=\left(\begin{matrix}1&0\\0 &1\end{matrix}\right)
$, the $Q_j$ functions are
\begin{align}
    Q_1&=|z_1|^2(1-|f_2|^2-|f_3|^2)+|z_2|^2(1-|f_1|^2-|f_2|^2)+2\Re(z_1\overline{z_2}(\overline{f_2}f_3+\overline{f_1}f_2))+\\
    &+\Re(z_1^2(f_3\overline{(f_1f_3-f_2^2)}+\overline{f_1})+2z_1z_2(f_2\overline{(f_1f_3-f_2^2)}+\overline{f_2})+z_2^2(f_1\overline{(f_1f_3-f_2^2)}+\overline{f_3})), \\
    Q_2&=1-|f_1|^2-2|f_2|^2-|f_3|^2+|f_1f_3-f_2^2|^2.
\end{align}
In the case $\mathbf{H}=\left(\begin{matrix}1&0\\0 &-1\end{matrix}\right)$, the $Q_j$ functions are
\begin{align}
    Q_1&=|z_1|^2(1+|f_2|^2-|f_3|^2)+|z_2|^2(-1+|f_1|^2-|f_2|^2)+2\Re(z_1\overline{z_2}(\overline{f_2}f_3-\overline{f_1}f_2))+\\
    &+\Re(z_1^2(-f_3\overline{(f_1f_3-f_2^2)}+\overline{f_1})+2z_1z_2(f_2\overline{(f_1f_3-f_2^2)}-\overline{f_2})+z_2^2(-f_1\overline{(f_1f_3-f_2^2)}+\overline{f_3})), \\
    Q_2&=1-|f_1|^2+2|f_2|^2-|f_3|^2+|f_1f_3-f_2^2|^2
\end{align}
In the case $\mathbf{H}=\left(\begin{matrix}0&1\\1 &0\end{matrix}\right)$, the $Q_j$ functions are
\begin{align}    Q_1&=2|z_1|^2\Re(f_2\overline{f_3})+2|z_2|^2\Re(f_2\overline{f_1})+2\Re(z_1\overline{z_2}(1-\overline{f_1}f_3-|f_2|^2))+\\
    &+\Re(z_1^2(-f_3\overline{(f_1f_3-f_2^2)}+\overline{f_3})+2z_1z_2(f_2\overline{(f_1f_3-f_2^2)}+\overline{f_2})+z_2^2(-f_1\overline{(f_1f_3-f_2^2)}+\overline{f_1})), \\
    Q_2&=1-2\Re(f_1\overline{f_3})-2|f_2|^2+|f_1f_3-f_2^2|^2.
\end{align}
When considering the distinguished functions $f_j$, these do not provide defining equations in normal form, in general. Therefore, we apply biholomorphic transformations to $\zeta$ to simplify \eqref{new S2}, but bring it to normal form only if this really simplifies the defining equation. For each of these $2$-nondegenerate models, we use the Section \ref{Symbols of CR hypersurfaces} formulas to compute their bigraded symbols and modified symbols and we determine if they have constant modified symbols with respect to the normalization conditions in Table \ref{normalization conditions Table}. We omit recording most formulas from these computations because they are obtrusively long, and instead report just the results.  The computations are, however, produced within this article's supplementary Maple files \cite{anc_files} in full detail.

\subsection{Formulas of Table 3, row 1}
For modified symbols of row 1 in Table \ref{7-d. symbol table}, we use the defining equations corresponding to $\mathbf{H}=\left(\begin{matrix}1&0\\0 &\epsilon\end{matrix}\right)$ with $\epsilon=\pm1$. For $\tau>0$ in these cases, computing \eqref{new S2} with the simplifying transformation $\zeta\mapsto\frac{\zeta}{\tau\sqrt{\lambda}}$ yields
\[
f_1(\zeta)=\frac{\left(\lambda^2+1\right)\zeta-\left(\lambda^2-1\right)\cos(\zeta)\sin(\zeta)}{2\tau\sqrt{\lambda}},
\]
\[
f_2(\zeta)=\frac{\sin(\zeta)^2\left(\lambda^2-1\right)}{2\tau\lambda}
\]
and
\[
f_3(\zeta)=\frac{\left(\lambda^2+1\right)\zeta+\left(\lambda^2-1\right)\cos(\zeta)\sin(\zeta)}{2\tau\sqrt{\lambda^3}}.
\]
Clearly, a transformation bringing the above to a normal form would involve a local inverse of $f_1$ and would not simplify the defining equation.

For $\tau=0$, computing \eqref{new S2} directly (i.e., without transforming $\zeta$) yields the normal form with
\begin{align}\label{nonregular family fj values}
f_{1}(\zeta)=\zeta,\quad f_2(\zeta)=0\quad\mbox{ and }
\quad f_3(\zeta)=\lambda \zeta.
\end{align}
For all $\lambda>1$ and $\tau\geq 0$ these structures turn out to have non-constant bigraded symbols.

\subsection{Formulas of Table 3, row 2}
For modified symbols of row 2 in Table \ref{7-d. symbol table}, we use the defining equations corresponding to $\mathbf{H}=\left(\begin{matrix}0&1\\1 &0\end{matrix}\right)$. For $\tau>0$ in these cases, computing \eqref{new S2} with the simplifying transformation $\zeta\mapsto\frac{\zeta}{\tau e^{i\theta/2}}$ yields
\[
f_1(\zeta)=\frac{e^{-i\frac{\theta}{2}}}{2\tau}\left(\zeta(1+e^{2i\theta})+\sin(\zeta)\cos(\zeta)(1-e^{2i\theta})\right),
\quad
f_2(\zeta)=\frac{-i}{\tau}\sin(\theta)\sin(\zeta)^2
\]
and
\[
f_3(\zeta)=\frac{e^{i\frac{\theta}{2}}}{2\tau}\left(\zeta(1+e^{-2i\theta})+\sin(\zeta)\cos(\zeta)(1-e^{-2i\theta})\right).
\]
Again a transformation bringing the above to a normal form would involve a local inverse of $f_1$ and would not simplify the defining equation.

By calculating the matrices in \eqref{key example Xi mat}, which represent the bigraded symbol, one finds readily that this symbol is not constant if either $\theta\neq\frac{\pi}{2}$ or $\tau\neq\frac{1}{\sqrt{2}}$. The parameter settings $\theta=\frac{\pi}{2}$ and $\tau=\frac{1}{\sqrt{2}}$ turn out to be  the type \Rmnum{1} homogeneous model from Table \ref{table of defining equations}.

Hence the above choices for $f_1(\zeta)$, $f_2(\zeta)$, and $f_3(\zeta)$ with $\theta=\frac{\pi}{2}$ and $\tau=\frac{1}{\sqrt{2}}$ define a hypersurface equivalent to the type \Rmnum{1} model of Table \ref{table of defining equations}. There is however a preferable choice of coordinates in $\zeta$, namely we apply $\zeta\mapsto \frac{\ln{(2\zeta+1)}}{1-i}$ instead of the aforementioned transformation, as this change of coordinates achieves a rational defining equation. The resulting $f_j$ functions are
\[
f_1(\zeta)=\frac{(1+i)(\zeta+1)\zeta}{2\zeta+1},
\quad
f_2(\zeta)=\frac{\sqrt{2}i\zeta^2}{2\zeta+1},
\quad\mbox{ and }\quad
f_3(\zeta)=\frac{(i-1)(\zeta+1)\zeta}{2\zeta+1}.
\]
Let us emphasize that although bringing the above to a normal form nicely simplifies $f_1(\zeta)$ and $f_3(\zeta)$, transforming them to $\zeta$ and $i\zeta$, the transformed $f_2(\zeta)$ is much more complicated (and not even rational), so we find the displayed $f_j$ values above simpler. Using these simpler $f_j$ values in \eqref{gen def fun} produces exactly  \eqref{def eqn type 1 b}.

Lastly, for $\tau=0$, computing \eqref{new S2} directly (i.e., without transforming $\zeta$) yields the normal form with
\[
f_1(\zeta)=\zeta ,\quad
f_2(\zeta)=0,
\quad\mbox{ and }\quad
f_3(\zeta)=\zeta e^{i\theta}.
\]

For all of these structures, except for the type \Rmnum{1} model produced by $\tau=\frac{1}{\sqrt{2}}$ and $\theta=\frac{\pi}{2}$, their bigraded symbols are not constant.

\subsection{Formulas of Table 3, row 3}
Row 3 of Table \ref{7-d. symbol table} gives the modified symbols of the type \Rmnum{5}.A and \Rmnum{6} model structures from Table \ref{table of defining equations}, and we derive their defining equation formulas corresponding to $\mathbf{H}=\left(\begin{matrix}1&0\\0 &\epsilon\end{matrix}\right)$. Computing \eqref{new S2} for these yields the normal form with
\begin{align}\label{off-diag fjk values}
f_{1}(\zeta)=0,\quad f_2(\zeta)=\zeta\quad\mbox{ and }
\quad f_3(\zeta)=0.
\end{align}

\subsection{Formulas of Table 3, row 4}

For modified symbols of row 4 in Table \ref{7-d. symbol table}, we use the defining equations corresponding to $\mathbf{H}=\left(\begin{matrix}1&0\\0 &\epsilon\end{matrix}\right)$ with $\epsilon=\pm1$. Computing \eqref{new S2} directly yields the normal form with
\[
f_1=\zeta
,\quad 
f_2=\frac{\zeta^2\tau}{2},
\quad\mbox{ and }\quad
f_3=\frac{\zeta^3\tau^2}{3}.
\]
Let us note that this is special instance of the normal form in case (4) with $\mu=2$ from Proposition \ref{table4prop1}.

When $\tau=0$, these formulas yield the type \Rmnum{4}.A and \Rmnum{4}.B defining equations in Table \ref{table of defining equations}. In the mixed signature case (i.e., $\epsilon=-1$), if $\tau=\frac{\sqrt{3}}{2}$ then these formulas yield a structure equivalent to the Type \Rmnum{3} model in Table \ref{table of defining equations}, but to obtain \eqref{def eqn type 3} exactly, one has to first rescale these formulas by $\zeta\mapsto 2\zeta$, which brings it out of the normal form, but makes a nicer defining equation.

For all other values of $\tau$ (including the pseudoconvex case with $\tau=\frac{\sqrt{3}}{2}$), these structures have constant bigraded symbols but non-constant modified symbols, which can be directly computed using the formulas of Section \ref{Symbols of CR hypersurfaces}.

\subsection{Formulas of Table 3, row 5}
Row 5 of Table \ref{7-d. symbol table} gives the modified symbol of the type \Rmnum{5}.B model structure from Table \eqref{table of defining equations}, and we derive its defining equation using the general formula corresponding to $\mathbf{H}=\left(\begin{matrix}0&1\\1 &0\end{matrix}\right)$. Computing \eqref{new S2} yields a normal form again with the formulas in \eqref{off-diag fjk values}.

\subsection{Formulas of Table 3, row 6}
For modified symbols of row 6 in Table \ref{7-d. symbol table}, we use the defining equations corresponding to $\mathbf{H}=\left(\begin{matrix}0&1\\1 &0\end{matrix}\right)$. For $\tau>0$ in these cases, computing \eqref{new S2} with the simplifying transformation $\zeta\mapsto\frac{\ln(\zeta+1)}{2\tau}$ yields
\[
f_{1}(\zeta)=\frac{\zeta(1+\zeta)}{2\tau},
\quad
f_2(\zeta)=\frac{\zeta}{2\tau},
\quad\mbox{ and }\quad
f_3(\zeta)=0.
\]
Again a transformation bringing the above to a normal form would involve a local inverse of $f_1$ and would not simplify the defining equation.

When $\tau=\frac{1}{2}$, these formulas yield the exactly type \Rmnum{2} defining equation in Table \ref{table of defining equations}. Again the normal form of this homogeneous model differs from the simple realization of this modified symbol. 

For $\tau=0$, computing \eqref{new S2} directly (i.e., without transforming $\zeta$) yields a normal form with
\[
 f_1(\zeta)=f_2(\zeta)=\zeta \quad\mbox{ and }\quad f_3(\zeta)=0.
\]

 For all $\tau\neq\frac{1}{2}$, these structures have constant bigraded symbols but non-constant modified symbols, which can be directly computed using the formulas of Section \ref{Symbols of CR hypersurfaces}.

\subsection{Formulas of Table 3, row 7}
Row 7 of Table \ref{7-d. symbol table} gives the modified symbol of the type \Rmnum{7} model structure from Table \eqref{table of defining equations}, and we derive its defining equation using the general formula corresponding to $\mathbf{H}=\left(\begin{matrix}0&1\\1 &0\end{matrix}\right)$. Computing \eqref{new S2} directly yields a normal form with
\[
f_1(\zeta)=\zeta \quad\mbox{ and }\quad f_2(\zeta)=f_3(\zeta)=0.
\]
The resulting defining equation is the concise formula $\Re(w)=\Re(z_1\overline{z_2}+z_2^2\overline{\zeta})$. It is the only known polynomial defining equation for a $2$-nondegenerate homogeneous model in this dimension.

We can summarize the results in this Section in the following result.

\begin{proposition}\label{all hom models prop}
    The Table \ref{table of defining equations} contains all (up to local CR equivalence) homogeneous $2$-nondegenerate models in $\mathbb{C}^4.$
\end{proposition}
\begin{proof}
  In \cite{SykesHomogeneous}, it is shown which modified symbols admit a homogeneous model and that the model for each one is locally unique. In this section we found the defining equation for each of these homogeneous models.
\end{proof}

\section{Infinitesimal symmetry algebras of homogeneous models}\label{Symmetry Algebras}\label{Infinitesimal symmetry algebras of homogeneous models}

For the homogeneous models in Table \ref{table of defining equations}, their algebras of infinitesimal symmetries can all be computed, and in this section we describe these infinitesimal symmetries for the homogeneous models of type \Rmnum{1}, \Rmnum{2} and \Rmnum{3}. Infinitesimal symmetry algebra descriptions for the other homogeneous model types appear across the literature cited in Table \ref{table of defining equations}.

\subsection{Type \Rmnum{1} model symmetries}
The $2$-nondegenerate model $M_0$ given by \eqref{def eqn type 1 b} has an eight dimensional $\mathbb{Z}$-graded infinitesimal symmetry algebra represented by holomorphic vector fields $\mathfrak{hol}(M_0)=\mathfrak{a}_{-2}\oplus\mathfrak{a}_{-1}\oplus \mathfrak{a}_{0}$, with graded components 
\[
\mathfrak{a}_{-2}=\left\{2bi\frac{\partial}{\partial w}\right\},
\]
\[
\mathfrak{a}_{-1}=\left\{
\parbox{\textwidth-3.5cm}{\flushleft$\displaystyle
2\left(\overline{a_1}z_2+\overline{a_2}z_1\right)\frac{\partial}{\partial w}+\left(a_1-\overline{a_1}\frac{\sqrt{2}i\zeta^2}{2\zeta+1}-\overline{a_2}\frac{(1+i)(\zeta+1)\zeta}{2\zeta+1}\right)\frac{\partial}{\partial z_1}+$\\
$\displaystyle +\left(a_2-\overline{a_1}\frac{(i-1)(\zeta+1)\zeta}{2\zeta+1}-\overline{a_2}\frac{\sqrt{2}i\zeta^2}{2\zeta+1}\right)\frac{\partial}{\partial z_2}$
}
\right\},
\]
and
\[
\mathfrak{a}_0=
\left\{
\parbox{\textwidth-3.5cm}{\flushleft$\displaystyle
\left(2cw+\overline{d}\left(-(1+i)z_1^2+(1-i)z_2^2\right)\right)\frac{\partial}{\partial w}+d\left(2\zeta+1\right)\frac{\partial}{\partial \zeta}+$\\
$\displaystyle
+\left(cz_1+de^{-i\frac{\pi}{4}}z_2+2\overline{d}\frac{-e^{i\frac{\pi}{4}}\left(\zeta^2+\zeta+\frac{1}{2}\right)z_2+i(\zeta^2+\zeta)z_1}{2\zeta+1}\right)\frac{\partial}{\partial z_1}+$\\
$\displaystyle +\left(cz_2+de^{i\frac{\pi}{4}}z_1-2\overline{d}\frac{e^{-i\frac{\pi}{4}}\left(\zeta^2+\zeta+\frac{1}{2}\right)z_1+i(\zeta^2+\zeta)z_2}{2\zeta+1}\right)\frac{\partial}{\partial z_2}$
}
\right\},
\]
for $b,c\in\mathbb{R}$ and $a_1,a_2,d\in\mathbb{C}$.
In \cite{beloshapka2004universal}, having defining equations with polynomial symmetries is among the listed properties for models considered there, so it is notable that these infinitesimal symmetries are not represented by polynomial vector fields. The Type \Rmnum{1} model is the only $2$-nondegenerate model in Table \ref{table of defining equations} for which this happens.

\subsection{Type \Rmnum{2} model symmetries}
The $2$-nondegenerate model $M_0$ given by \eqref{def eqn type 2 b} has an eight dimensional $\mathbb{Z}$-graded infinitesimal symmetry algebra represented by holomorphic vector fields $\mathfrak{hol}(M_0)=\mathfrak{a}_{-2}\oplus\mathfrak{a}_{-1}\oplus \mathfrak{a}_{0}$, with graded components 
\[
\mathfrak{a}_{-2}=\left\{2bi\frac{\partial}{\partial w}\right\},
\]
\[
\mathfrak{a}_{-1}=\left\{
\parbox{\textwidth-3.5cm}{\flushleft$\displaystyle
2\left(\overline{a_2}z_1+\overline{a_1}z_2\right)\frac{\partial}{\partial w}+\left(a_1-\overline{a_1}\zeta-\overline{a_2}\zeta-\overline{a_2}\zeta^2\right)\frac{\partial}{\partial z_1}+\left(a_2-\overline{a_2}\zeta\right)\frac{\partial}{\partial z_2}$
}
\right\},
\]
and
\[
\mathfrak{a}_0=
\left\{
\parbox{\textwidth-2cm}{\flushleft$\displaystyle
\left(2cw+2\overline{d}(z_2^2+2z_1z_2)\right)\frac{\partial}{\partial w}+2(d-\overline{d}\zeta)(1+\zeta)\frac{\partial}{\partial \zeta}+\left(cz_2-2\overline{d}z_2(1+\zeta)\right)\frac{\partial}{\partial z_2}+$\\
$\displaystyle +\left(cz_1+d(z_2+2z_1)-\overline{d}z_2-2\overline{d}\zeta(2z_2+z_1+z_2\zeta)\right)\frac{\partial}{\partial z_1}$
}
\right\},
\]
for $b,c\in\mathbb{R}$ and $a_1,a_2,d\in\mathbb{C}$.

\subsection{Type \Rmnum{3} model symmetries}
The $2$-nondegenerate model $M_0$ given by \eqref{def eqn type 3} has a nine dimensional $\mathbb{Z}$-graded infinitesimal symmetry algebra represented by holomorphic vector fields $\mathfrak{hol}(M_0)=\mathfrak{a}_{-2}\oplus\mathfrak{a}_{-1}\oplus \mathfrak{a}_{0}$, with graded components 
\[
\mathfrak{a}_{-2}=\left\{2bi\frac{\partial}{\partial w}\right\},
\]
\[
\mathfrak{a}_{-1}=\left\{ 2\left(\overline{a_1}z_1-\overline{a_2}z_2\right)\frac{\partial}{\partial w}+\left(a_1-2\overline{a_1}\zeta+\sqrt{3}\overline{a_2}\zeta^2\right)\frac{\partial}{\partial z_1}+\left(a_2-\sqrt{3}\overline{a_1}\zeta^2+2\overline{a_2}\zeta^3\right)\frac{\partial}{\partial z_2}
\right\},
\]
\[
\mathfrak{a}_{0}=\left\{
\parbox{13.1cm}{\flushleft$\displaystyle
\left(2\Re(c_2)w+2\overline{c_1}z_1^2\right)\frac{\partial}{\partial w}+\left(c_2z_1+\sqrt{3}\overline{c_1}z_2-4\overline{c_1}z_1\zeta\right)\frac{\partial}{\partial z_1}+$\\
$\displaystyle +\left(\sqrt{3}c_1z_1+c_2z_2+2i\Im(c_2)z_2-2\sqrt{3}\overline{c_1}z_1\zeta^2\right)\frac{\partial}{\partial z_2}+\left(c_1-\overline{c_1}\zeta^2+2i\Im(c_2)\zeta\right)\frac{\partial}{\partial \zeta}$
}
\right\}
\]
for $b\in\mathbb{R}$ and $a_1,a_2,c_1,c_2\in\mathbb{C}$. These infinitesimal symmetries comprise a well known algebra, namely one of the maximal parabolic subalgebras in the exceptional Lie algebra $\mathrm{Lie}(G_2)$ \cite{SykesHomogeneous}. Specifically, it is the Lie algebra of the subgroup labeled as the \emph{$G_2$ contact parabolic} in \cite[Section 2.3]{leistner2017new}.

 \newcommand{\noop}[1]{}

\end{document}